\documentclass[a4paper,11pt,twoside]{amsart}

\usepackage{amsmath}
\usepackage{amsthm}
\usepackage{amsfonts, amssymb}
\usepackage{mathrsfs}
\usepackage[all]{xy}
\usepackage{latexsym}
\usepackage[dvips]{graphicx}
\usepackage{enumerate}

\theoremstyle{plain}
\newtheorem{lemma}{Lemma}[section]
\newtheorem{prop}[lemma]{Proposition}
\newtheorem{theo}[lemma]{Theorem}
\newtheorem{coro}[lemma]{Corollary}

\theoremstyle{remark}
\newtheorem{rem}[lemma]{Remark}

\newtheorem*{notat}{Notation}
\theoremstyle{definition}
\newtheorem{definition}[lemma]{Definition}
\newtheorem{ex}[lemma]{Example}

\def\id{\mathrm{Id}}

\def\inc{\mathrm{in}}
\def\incl{\mathrm{inc}}

\def\K{\mathcal{K}}

\def\Ima{\mathrm{Im}\:}
\def\sgn{\mathrm{sgn}}

\def\N{\mathbb{N}}
\def\Z{\mathbb{Z}}

\def\rg{\mathrm{rg}}

\def\tq{\;/\;}
\def\ch{\mathrm{Ch}}

\begin{document}

\title{On homology of finite topological spaces}

\author[N. Cianci]{Nicol\'as Cianci}
\email{nicocian@gmail.com}

\author[M. Ottina]{Miguel Ottina}
\email{emottina@uncu.edu.ar}
\thanks{Research partially supported by grant M015 of SeCTyP, UNCuyo.}

\address{Facultad de Ciencias Exactas y Naturales \\
Universidad Nacional de Cuyo \\ Mendoza, Argentina.}

\begin{abstract}
We develop a new method to compute the homology groups of finite topological spaces (or equivalently of finite partially ordered sets) by means of spectral sequences giving a complete and simple description of the corresponding differentials. Our method proves to be powerful and involves far fewer computations than the standard one. We derive many applications of our technique which include a generalization of Hurewicz theorem for regular CW-complexes, results in homological Morse theory and formulas to compute the M\"obius function of posets.
\end{abstract}


\maketitle

\section{Introduction}

The interaction between topology and combinatorics has proved to be very fruitful. Examples of this interaction are simplicial homology, discrete Morse theory \cite{For, Min}, the celebrated proof of Kneser's conjecture given by Lov\'asz \cite{Lov}, the subsequent developments in the study of graph properties by means of topological methods \cite{Koz} and the theory of finite topological spaces, which has grown considerably in the last years from works by Barmak and Minian \cite{BarLN, BM1, BM2, BM3, BM4}.

The theory of finite topological spaces is based in the well-known correspondence between finite posets and finite $T_0$--spaces given by Alexandroff \cite{Alex} and in the works of Stong \cite{Sto} and McCord \cite{McC} who study finite spaces from totally different perspectives. Stong studies the homotopy types of finite topological spaces by means of an elementary move which consists of removing a single point of a finite space. Surprisingly, a sequence of these simple moves is enough to determine whether two given finite spaces have the same homotopy type. On the other hand, McCord establishes a correspondence that assigns to each finite $T_0$--space $X$ a simplicial complex $\mathcal{K}(X)$ together with a weak homotopy equivalence $\mathcal{K}(X)\to X$ and proves that the weak homotopy types of compact polyhedra are in one to one correspondence with the weak homotopy types of finite topological spaces.

Barmak and Minian delve deeply into this theory and obtain many interesting results, among which we mention the introduction of an elementary move in finite $T_0$--spaces which corresponds exactly with the elementary collapses of simple homotopy theory of compact polyhedra \cite{BM1} and a generalization of McCord's result on the weak equivalence between a compact polyhedron and its order complex \cite{BM2}. Moreover, they use the theory of finite spaces to study Quillen's conjecture on the poset of non-trivial $p$--subgroups of a group and Andrews--Curtis' conjecture (see \cite{BarLN}).

As it is shown by the recent works of Barmak and Minian, finite topological spaces can be used in several different situations to develop new tools and techniques to study topological and combinatorial problems. Moreover, in many of them the finite space approach is simpler, more adequate or more tractable than the one given by simplicial complexes and polyhedra. In a similar way, problems regarding homotopy invariants of finite topological spaces, which can be tackled by the simplicial complex approach, can be dealt with in a more direct and natural way in their own context.

It is this idea which is exploited in this article, where we develop a new method to compute the homology groups of finite topological spaces by means of spectral sequences. Not only do we give spectral sequences which converge to the homology groups of a given finite space but also we describe completely the differentials of all the pages of those spectral sequences. Our method proves to be powerful and involves far fewer computations than the standard one of computing the simplicial homology groups of the order complex of the finite space. Moreover, it can be applied to Alexandroff spaces provided that a suitable filtration exists.

As an application of this result, we give a spectral sequence which converges to the homology groups of the universal cover of a locally finite T$_0$--space which gives a method to compute the second homotopy group of such spaces. With this result we obtain a generalization of Hurewicz's theorem for a class of regular CW-complexes which are not necessarily simply-connected.

Moreover, we apply our results to generalize a result of Minian on homological Morse theory for posets \cite{Min}, largely extending the class of posets for which it is valid, with a very nice and conceptual proof. And we also apply our techniques to obtain different formulas to compute the M\"obius function of posets which include an alternative proof to a result of Bj\"orner and Walker \cite{BjoWal}.

\section{Preliminaries} \label{sect_prelim}

If $X$ is a finite $T_0$ topological space and $x\in X$, $U_x$ denotes the minimal open set which contains $x$, that is, the intersection of all the open sets of $X$ which contain $x$. In a similar way, $F_x$ denotes the minimal closed set which contains $x$, that is, $F_x=\overline{\{x\}}$.

An order relation can be defined in a finite $T_0$--space $X$ as follows: $x\leq y$ if and only if $U_x\subseteq U_y$. Conversely, if $P$ is a finite poset then the subsets $\{x\in P \tq x\leq a\}$, $a\in P$, form a basis for a topology on $P$. These applications are mutually inverse and give a one-to-one correspondence between finite $T_0$--spaces and finite posets \cite{Alex} (see also \cite{BarLN}). This correspondence extends to a one-to-one correspondence between Alexandroff $T_0$--spaces and posets \cite{McC}. 

Hence, from now on, we will see any Alexandroff $T_0$--space as a poset and any poset as an Alexandroff $T_0$--space without further notice. Note that locally finite topological spaces are Alexandroff spaces.

If $X$ is an Alexandroff $T_0$--space then $U_x=\{a\in X \tq a\leq x\}$ and $F_x=\{a\in X \tq a\geq x\}$. It is standard to define
$\hat{U}_x=\{a\in X \tq a< x\}$, $\hat{F}_x=\{a\in X \tq a> x\}$, $C_x=U_x\cup F_x$ and $\hat{C}_x=C_x-\{x\}$. In case several topological spaces are considered at the same time, we will denote $U_x$ by $U_x^X$ to indicate the space in which the minimal open set is considered. We will use similar notations for $F_x$, $C_x$, $\hat{U}_x$, $\hat{F}_x$ and $\hat{C}_x$.

Let $X$ be a finite $T_0$--space and let $x\in X$. The point $x$ is an \emph{up beat point} of $X$ if the subposet $\hat{F}_x$ has a minimum. The point $x$ is a \emph{down beat point} of $X$ if the subposet $\hat{U}_x$ has a maximum. The point $x$ is a \emph{beat point} of $X$ if it is either an up beat point or a down beat point. Stong proves in \cite{Sto} that if $x$ is a beat point of $X$ then $X-\{x\}$ is a strong deformation retract of $X$. Moreover, he gives a simple criterion to decide whether two given finite topological spaces are homotopy equivalent. Using the results of Stong it is easy to prove that if $X$ is a finite $T_0$--space and $x\in X$ then $C_x$ is contractible.

The \emph{order complex} of an Alexandroff $T_0$--space $X$ is the simplicial complex $\mathcal{K}(X)$ of the finite non-empty chains of $X$. McCord proves in \cite{McC} that there exists a weak homotopy equivalence $|\mathcal{K}(X)| \to X$ ($|\mathcal{K}(X)|$ denotes the geometric realization of $\mathcal{K}(X)$). He also proves in \cite{McC} that there exists a correspondence that assigns to each Alexandroff space $Z$ an Alexandroff $T_0$--space $\hat{Z}$ which is a quotient of $Z$ and which satisfies that the quotient map $Z\to\hat{Z}$ is a homotopy equivalence.

Also, the \emph{face poset} of a simplicial complex $K$ is the poset $\mathcal{X}(K)$ of simplices of $K$ ordered by inclusion. Clearly, $\mathcal{K}(\mathcal{X}(K))$ is the barycentric subdivision of $K$. This leads to the definition of the \emph{barycentric subdivision} of a poset $X$ as $X'=\mathcal{X}(\mathcal{K}(X))$. There is a weak homotopy equivalence $X'\to X$ which takes each nonempty chain of $X$ to its maximum \cite{Wal}.

The \emph{non-Hausdorff suspension} of a topological space $X$ is the space $\mathbb{S}(X)$ whose underlying set is $X\cup \{+,-\}$ and whose open sets are those of $X$ together with $X\cup\{+\}$, $X\cup\{-\}$ and $X\cup\{+,-\}$. This definition was introduced by McCord in \cite{McC}, where he proves that for every space $X$ there exists a weak homotopy equivalence between the suspension of $X$ and $\mathbb{S}(X)$.

A \emph{finite model} of a topological space $Z$ is a finite space which is weak homotopy equivalent to $Z$. For example, if $D_2$ is the discrete space of two points and $n\in\N$ then $\mathbb{S}^n D_2$ is a finite model of the $n$--sphere $S^n$.

If $X$ is a poset $X^\textnormal{op}$ will denote the poset $X$ with the inverse order.

Recall that a poset is \emph{homogeneous} of dimension $n$ if all its maximal chains have cardinality $n+1$. 

A poset $X$ is \emph{graded} if $U_x$ is homogeneous for all $x\in X$. In this case, the degree of $x$ is the dimension of $U_x$ and is denoted by $\deg(x)$.

The following definitions were introduced by Minian in \cite{Min}.

\begin{itemize}
\item A finite poset $X$ is called \emph{$h$--regular} if for every $x\in X$, the order complex of $\hat U_x$ is homotopy equivalent to $S^{n-1}$ where $n$ is the maximum of the cardinality of the chains in $\hat U_x$.
\item A \emph{cellular} poset is a graded poset $X$ such that for every $x\in X$, $\hat U_x$ has the homology of a $(p-1)$--sphere, where $p=\deg(x)$.
\end{itemize}

In a similar way we will say that a finite $T_0$--space is \emph{cellular} if its associated poset is cellular. Note that the barycentric subdivision of a poset is a cellular poset.

If $n\in\N_0$, an $n$--chain of $P$ is a chain of $P$ of cardinality $n+1$. The empty chain will be regarded as a $(-1)$--chain. We will use the notation $[v_0,\dots,v_n]$ for an $n$--chain $\{v_0,\dots,v_n\}$ of $P$ with $v_{j-1}<v_{j}$ for all $j\in\{1,2,\ldots,n\}$. Also, if $n\in\N_0$ and $s=[v_0,\dots,v_n]$ is an $n$--chain and $k\in\{0,\ldots,n\}$ then $s_{\hat k}$ will denote the $(n-1)$--chain $[v_0,\dots,\hat{v_k},\dots,v_n]$.

\begin{notat}
If $X$ is a poset, $\ch(X)$ will denote the set of chains of $X$ and, for $n\in\N_0$, $\ch_n(X)$ will denote the set of $n$--chains of $X$.
\end{notat}

From McCord's theorem it is clear that $H_n(X)=H_n(|\mathcal{K}(X)|)$. Hence, the homology groups of a finite topological space can be computed from the simplicial chain complex associated to $\mathcal{K}(X)$. This fact can be expressed entirely in terms of the finite space $X$ since the simplices of $\mathcal{K}(X)$ are the chains of $X$. Thus, making the translation to the context of posets, we can introduce the following definition which will be useful for our work.

\begin{definition}
Let $X$ be a finite $T_0$--space. The \emph{$f$--chain complex associated to $X$} is the chain complex $C^f(X)=(C^f_n(X),d^f_n)_{n\in\Z}$ defined by
\begin{displaymath}
C^f_n(X)=\left\{ 
\begin{array}{cl} 
\bigoplus\limits_\textnormal{$\ch_n(X)$} \Z & \textnormal{if $n\geq 0$} \\
0 & \textnormal{if $n<0$}
\end{array}
\right.
\end{displaymath}
and where, for $n\in \N$, the morphisms $d^f_n:C^f_n(X)\longrightarrow C^f_{n-1}(X)$ are defined by
$$d^f_n([v_0,\dots,v_n])=\sum\limits_{i=0}^n (-1)^i [v_0,\dots,\hat{v_i},\dots,v_n]$$
for all $[v_0,\dots,v_n]\in \ch_n(X)$.
\end{definition}

It is clear that the chain complex $C^f(X)$ is isomorphic to the simplicial chain complex of $\K(X)$.

\begin{definition}
Let $X$ be a finite $T_0$--space. For $n\in \Z$ we define $H^f_n(X)$ as the $n$--th homology group of the chain complex $C^f(X)$.
The group $H^f_n(X)$ will be called the \emph{$n$--th $f$--homology group of $X$}.
\end{definition}

It follows that if $X$ is a finite $T_0$--space then $H_n(X)=H^f_n(X)$ for all $n\in \Z$.

A similar translation can be made for the relative case. We include below the notations and definitions for this case so as to set the terminology that we will use.

\begin{definition}
Let $X$ be a finite $T_0$--space and let $A\subseteq X$. An \emph{$n$--chain of $(X,A)$} is an $n$--chain of $X$ which is not included in $A$. The set of $n$--chains of $(X,A)$ will be denoted by $\ch_n(X,A)$.
\end{definition}

\begin{definition}
Let $X$ be a finite $T_0$--space and let $A\subseteq X$. We define the \emph{$f$--relative chain complex associated to $(X,A)$} as the chain complex $C^f(X,A)=C^f(X)/C^f(A)$. We also define, for $n\in\Z$, $H^f_n(X,A)$ as the $n$--th homology group of the chain complex $C^f(X,A)$.
\end{definition}

Note that, for $n\in \N_0$, $C^f_n(X,A)\cong \bigoplus\limits_\textnormal{$\ch_n(X,A)$} \Z$.

Clearly, the chain complex $C^f(X,A)$ is isomorphic to the relative simplicial chain complex of $(\K(X),\K(A))$ and hence $H^f_n(X,A)=H_n(X,A)$ for all $n\in\Z$.

In a similar way, one can define reduced $f$--homology groups. In this case, the group in degree $-1$ of the corresponding chain complex will be the free abelian group generated by the empty chain.

\section{First main theorem} \label{main_section}

In this section we will develop a spectral sequence which converges to the homology groups of a given finite space and we will provide an explicit description of the differentials of all the pages of this spectral sequence. Moreover, we will prove that our result generalizes a cellular-type method of Minian \cite{Min}.

\begin{notat}
Let $X$ be a finite $T_0$--space, let $A\subseteq X$ and let $n\in \N_0$. If $\sigma\in C^f_n(X)$ then we will write $\overline{\sigma}^A$ (or simply $\overline{\sigma}$) for the class of $\sigma$ in $C^f_n(X,A)$. 

If $x\in X$ and $\sigma\in C^f_n(C_x)$, we will write $\overline{\sigma}^x$ for the class of $\sigma$ in $C^f_n(C_x,\hat{C}_x)$.

The class of the $n$--cycle $\sigma$ in $H^f_n(X)$ will be denoted by $[\sigma]$, and the class of the $n$--relative cycle $\overline{\sigma}^A$ in $H^f_n(X,A)$ will be denoted by $[\overline{\sigma}^A]$.
\end{notat}

If $I$ is a set, $H$ is a group, $\{G_i\}_{i\in I}$ is a collection of groups and $\{f_i:G_i\longrightarrow H\}_{i\in I}$ is a collection of group homomorphisms, we define $\biguplus\limits_{i\in I}f_i:\bigoplus\limits_{i\in I}G_i\longrightarrow H$ by $$\left(\biguplus\limits_{i\in I}f_i\right)((x_i)_{i\in I})=\sum\limits_{i\in I}f_i(x_i).$$

The following lemma, which will be used to prove proposition \ref{prop:+ hom C_x}, contains a simple idea which will be very important in our method.

\begin{lemma}\label{lemma:isos phi}
Let $X$ be a finite $T_0$--space and let $D\subseteq X$ be an antichain. For $x\in D$ and $n\in \N_0$ we define:
\begin{itemize}
\item $i^x_n:C^f_n(C_x,\hat{C}_x)\longrightarrow C^f_n(X,X-D)$ by $i^x_n(\overline{\sigma}^x)=\overline{\sigma}^{X-D}$ for every $\sigma\in C^f_n(C_x)$.
\item $\rho^x_n:C^f_n(X,X-D)\longrightarrow C^f_n(C_x,\hat{C}_x)$ as the group homomorphism that satisfies
$$\rho^x_n(\overline{s}^{X-D})=
\left\{
\begin{array}{ll}
\overline{s}^x& \text{if $x\in s$}\\
0& \text{if $x\not \in s$}\\
\end{array}
\right.
$$
for every $n$--chain $s$ in $X$.
\item $\phi_n:C^f_n(X,X-D)\longrightarrow \bigoplus\limits_{x\in D}C^f_n(C_x,\hat{C}_x)$  as the group homomorphism that satisfies
$$\phi_n(\overline{s}^{X-D})=(\rho^x_n(\overline{s}^{X-D}))_{x\in D}$$
for every $n$--chain $s$ in $X$.
\end{itemize}

Then $\phi_n$ is a group isomorphism and $\phi^{-1}_n=\biguplus\limits_{x\in D}i^x_n$ for every $n\in \N_0$.
\end{lemma}

\begin{proof}
It is easy to check that the morphisms above are well-defined.

For each $n\in \N_0$ and for each $x_0\in D$, let $\inc^{x_0}_n:C^f_n(C_{x_0},\hat{C}_{x_0})\longrightarrow \bigoplus\limits_{x\in D}C^f_n(C_x,\hat{C}_x)$ be the canonical inclusion.

Let $n\in \N_0$. It is clear that $\biguplus\limits_{x\in D}i^x_n$ is an epimorphism since every $n$--chain of $(X,X-D)$ is an $n$--chain of $(C_x,\hat{C}_x)$ for some $x\in D$. Besides, $\phi_n i^x_n(s)=\inc^x_n(s)$ for every $x\in D$ and for every $n$--chain $s$ in $(C_x,\hat{C}_x)$ since every $n$--chain of $(X,X-D)$ must have exactly one element of $D$, as $D$ is an antichain. Therefore, $\phi_n i^x_n=\inc^x_n$.

Then, $\phi_n\circ\biguplus\limits_{x\in D}i^x_n=\biguplus\limits_{x\in D}(\phi_n i^x_n)=\biguplus\limits_{x\in D}\inc^x_n=\id$. Hence, $\biguplus\limits_{x\in D}i^x_n$ is a monomorphism. Thus, $\biguplus\limits_{x\in D}i^x_n$ is an isomorphism with inverse $\phi_n$.
\end{proof}

\begin{prop}\label{prop:+ hom C_x}
Let $X$ be a finite $T_0$--space and let $D$ be an antichain in $X$. Then $H_n(X,X-D)\cong\bigoplus\limits_{x\in D}\tilde{H}_{n-1}(\hat{C}_x)$ for every $n\in \Z$.
\end{prop}

\begin{proof}
For each $n\in \Z$ we have a commutative diagram
\[
\xymatrix@C=80pt@R=60pt{C^f_n(X,X-D)\ar[r]^{\overline{d}^f_n}&C^f_{n-1}(X,X-D)\\
\bigoplus\limits_{x\in D}C^f_n(C_x,\hat{C}_x)\ar[u]^{\biguplus\limits_{x\in D}i^x_n}\ar[r]_{\bigoplus\limits_{x\in D}(\overline{d}^f_n)_x}&\bigoplus\limits_{x\in D}C^f_{n-1}(C_x,\hat{C}_x)\ar[u]^{\biguplus\limits_{x\in D}i^x_{n-1}}}
\]
where for every $x\in D$ the morphism $(\overline{d}^f_n)_x$ is the restriction of $\overline{d}^f_n$ to $C^f_n(C_x,\hat{C}_x)$.

Then, the chain complexes $C^f(X,X-D)$ and $\bigoplus\limits_{x\in D}C^f(C_x,\hat{C}_x)$ are isomorphic and therefore
$$H_n(X,X-D) \cong H^f_n(X,X-D) \cong \bigoplus\limits_{x\in D}H^f_n(C_x,\hat{C}_x) \cong \bigoplus\limits_{x\in D}H_n(C_x,\hat{C}_x)$$ for every $n\in \Z$.

Now, since $C_x$ is contractible, $H_n(C_x,\hat{C}_x) \cong \tilde{H}_{n-1}(\hat{C}_x)$. The result follows.
\end{proof}

\begin{rem}
In the previous lemma we allow $\hat{C}_x$ to be empty, in which case $\widetilde{H}_{-1}(\hat{C}_x)=\Z$ and $\widetilde{H}_n(\hat{C}_x)=0$ for $n\neq -1$.
\end{rem}

\begin{definition}
Let $X$ be a finite $T_0$--space. Let $n\in \N_0$ and let $x\in X$. Let $s=[y_0,\dots,y_n]$ be an $n$--chain in $(C_x,\hat{C}_x)$. Let $k^s_x$ denote the only integer $i\in \{0,\dots,n\}$ such that $y_i=x$. We define the \emph{sign} of $x$ in $s$ by $\sgn_s(x)=(-1)^{k^s_x}$.
\end{definition}

\begin{lemma}\label{lemma:partial}
Let $X$ be a finite $T_0$--space, let $x\in X$ and let $n\in \N_0$.

Let $\partial:H_n(C_x,\hat{C}_x)\longrightarrow \tilde{H}_{n-1}(\hat{C}_x)$ be the connection homomorphism of the long exact sequence associated to the finite chain complex of $(C_x,\hat{C}_x)$.

Let $\sigma=\sum\limits_{i=1}^l \alpha_i s_i+\sum\limits_{j=1}^m \beta_j t_j\in C^f_n(C_x)$, where $l,m\in \N$, $\alpha_i\in \Z$ for every $i=1,\dots,l$, $\beta_i\in \Z$ for every $j=1,\dots,m$ and where for every $i=1,\dots,l$, $s_i$ is an $n$--chain in $C_x$ such that $x\in s_i$, and for every $j=1,\dots,m$, $t_j$ is an $n$--chain of $C_x$ such that $x\not\in t_j$.

If $\overline{\sigma}\in \ker\overline{d}^f_n$, then $\partial([\overline{\sigma}])=\left[\sum\limits_{i=1}^l \alpha_i \sgn_{s_i}(x)(s_i-\{x\})\right]$.
\end{lemma}

\begin{proof}
Let $\tau=\sum\limits_{i=1}^l \alpha_i s_i$. Note that
$$\overline{\sigma}=\overline{\sum\limits_{i=1}^l \alpha_i s_i}+\overline{\sum\limits_{j=1}^m \beta_j t_j}=\overline{\sum\limits_{i=1}^l \alpha_i s_i}=\overline{\tau}$$
in $C^f_n(C_x,\hat{C}_x)$ since $\sum\limits_{j=1}^m \beta_j t_j\in C^f_n(\hat{C}_x)$.

By the proof of the Snake Lemma, $d^f_n(\tau)\in C^f_{n-1}(\hat{C}_x)$ and $\partial([\overline{\sigma}])=\partial([\overline{\tau}])$ is the class of $d^f_n(\tau)$ in $H^f_{n-1}(\hat{C}_x)$. On the other hand
$$d^f_n(\tau)=\sum\limits_{i=1}^l\alpha_i \left(\sum\limits_{k=0}^n (-1)^k (s_i)_{\hat{k}}\right)=\sum\limits_{i=1}^l \alpha_i \sgn_{s_i}(x)(s_i-\{x\})+\sum\limits_{i=1}^l\alpha_i \left(\sum\limits_{k\neq k^{s_i}_x} (-1)^k (s_i)_{\hat{k}}\right).$$

Now, note that $\sum\limits_{i=1}^l \alpha_i\sgn_{s_i}(x)(s_i-\{x\})\in C^f_{n-1}(\hat{C}_x)$, since it is a sum of $(n-1)$--chains in $C_x$ that do not contain $x$. Since $d^f_n(\tau)\in C^f_{n-1}(\hat{C}_x)$, then $$\sum\limits_{i=1}^l\alpha_i \left(\sum\limits_{k\neq k^{s_i}_x} (-1)^k (s_i)_{\hat{k}}\right)\in C^f_{n-1}(\hat{C}_x)$$ 
But for every $i=1,\dots,l$, $k\neq k^{s_i}_x$ implies that $x\in(s_i)_{\hat{k}}$. Then, $\sum\limits_{k\neq k^{s_i}_x} (-1)^k (s_i)_{\hat{k}}$ is a sum of chains that contain $x$. Since $C^f_{n-1}(C_x)$ is a free abelian group, $\sum\limits_{i=1}^l \alpha_i\left(\sum\limits_{k\neq k^{s_i}_x} (-1)^k (s_i)_{\hat{k}}\right)=0$. Hence, 
$$d^f_n(\tau)=\sum\limits_{i=1}^l\alpha_i\sgn_{s_i}(x)(s_i-\{x\}),$$
and therefore,
$$\partial([\overline{\sigma}])=[d^f_n(\tau)]=\left[\sum\limits_{i=1}^l \alpha_i \sgn_{s_i}(x)(s_i-\{x\})\right].$$
\end{proof}

\begin{definition}
Let $X$ be a finite $T_0$--space and let $\mathscr{F}=\{X_p:p\in \Z\}$ be a filtration of $X$. We say that the filtration $\mathscr{F}$ is \emph{induced by antichains} if $X_{-1}=\varnothing$ and $X_n-X_{n-1}$ is an antichain for every $n\in \N$.

Note that the subposet $X_0$ needs not be an antichain.
\end{definition}

\bigskip

The following is one of the main theorems of this article.

\begin{theo} \label{theo:principal}
Let $X$ be a finite $T_0$--space and let $\{X_p:p\in \Z\}$ be a filtration of $X$ which is induced by antichains.
For each $p\in \N$, let $D_p=X_p-X_{p-1}$. 

Then there is a spectral sequence $\{(E^r_{p,q})_{p,q\in \Z},(d^r_{p,q})_{p,q\in \Z}\}_{r\in \N}$ that converges to $H_*(X)$ such that:
\begin{itemize}
\item $E^1_{p,q}=0$ for every $p\leq -1$.
\item $E^1_{0,q}=H_q(X_0)$.
\item $E^1_{p,q}=\bigoplus\limits_{x\in D_p}\tilde{H}_{p+q-1}(\hat{C}^{X_p}_x)$ for $p\geq 1$.
\item The morphisms $d^1_{p,q}:E^1_{p,q}\longrightarrow E^1_{p-1,q}$ are defined in the following way:
\begin{itemize}
\item If $p\leq 0$ and $q\in \Z$, then $d^1_{p,q}$ is the trivial homomorphism.
\item If $p=1$ and $q\in \N_0$, then $d^1_{p,q}:\bigoplus\limits_{x\in D_1}\tilde{H}_q(\hat{C}^{X_1}_x)\longrightarrow H_q(X_0)$ is defined by $$d^1_{1,q}(([\sigma_x])_{x\in D_1})=\sum\limits_{x\in D_1}[\sigma_x]$$.
\item If $p\geq 1$ and $q\leq -p$, then $d^1_{p,q}$ is the trivial homomorphism.
\item If $p\geq 2$ and $q\geq 1-p$, then $d^1_{p,q}:\bigoplus\limits_{x\in D_p}\tilde{H}_{p+q-1}(\hat{C}^{X_p}_x)\longrightarrow\bigoplus\limits_{y\in D_{p-1}}\tilde{H}_{p+q-2}(\hat{C}^{X_{p-1}}_y)$ is defined by
$$d^1_{p,q}\left(\left(\left[\sum\limits_{i=1}^{l_x}a^x_i s^x_i\right]\right)_{x\in D_p}\right)=\left(\left[\sum\limits_{x\in D_p}\sum\limits_{s^x_i\ni y}a^x_i\sgn_{s^x_i}(y)(s^x_i-\{y\})\right]\right)_{y\in D_{p-1}}$$ 
where for every $x\in D_p$, $l_x\in \N$, and for every $i=\{1,\dots,l_x\}$, $a^x_i\in \Z$ and $s^x_i\in \tilde{C}_{p+q-1}(\hat{C}^{X_p}_x)$.
\end{itemize}
\end{itemize}
\end{theo}

\begin{proof}
Let $\{(\tilde{E}^r_{p,q})_{p,q\in \Z},(\tilde{d}^r_{p,q})_{p,q\in \Z}\}_{r\in \N}$ be the bigraded spectral sequence associated to the filtration $\{X_p\}_{p\in \Z}$ of $X$, that is,
\begin{itemize}
\item $\tilde{E}^1_{p,q}=H_{p+q}(X_p,X_{p-1})$ for every $p,q\in \Z$ 
\item $\tilde{d}^1_{p,q}=j_* \partial$
\end{itemize}
where $j_*$ is the homomorphism induced in the homology groups by the projection $j:C^f_{p+q-1}(X_{p-1})\longrightarrow C^f_{p+q-1}(X_{p-1},X_{p-2})$ and where $\partial:H_{p+q}(X_p,X_{p-1})\longrightarrow \tilde{H}_{p+q-1}(X_{p-1})$ is the connection homomorphism of the long exact sequence associated to the pair $(X_p,X_{p-1})$.

Since $X$ is finite and $X_p=\varnothing$ for every $p\leq -1$, it follows that the spectral sequence $\{(\tilde{E}^r_{p,q})_{p,q\in \Z},(\tilde{d}^r_{p,q})_{p,q\in \Z}\}_{r\in \N}$ converges to $H_*(X)$.

For $p,q\in \Z$ we define:
\begin{itemize}
\item $E^1_{p,q}=0$ if $p\leq -1$.
\item $E^1_{0,q}=H_q(X_0)$.
\item $E^1_{p,q}=\bigoplus\limits_{x\in D_p}\tilde{H}_{p+q-1}(\hat{C}^{X_{p}}_x)$ if $p\geq 1$.
\end{itemize}

Since the filtration $\{X_p:p\in \Z\}$ is induced by antichains, $D_p$ is an antichain for every $p\in \N$. Thus, by the proof of \ref{prop:+ hom C_x}, for each $p\geq 1$ and for all $q\in\Z$ we have isomorphisms
$$\left(\biguplus\limits_{x\in D_p}i^x_n\right)_\ast\circ\left(\bigoplus\limits_{x\in D_p}\partial^x\right)^{-1}:\bigoplus\limits_{x\in D_p}\tilde{H}_{p+q-1}(\hat{C}^{X_{p}}_x)\longrightarrow H_{p+q}(X_p,X_{p-1})$$
where $\partial^x:H_{p+q}(C_x^{X_p},\hat C_x^{X_p})\to \tilde{H}_{p+q-1}(\hat C_x^{X_p})$ is the connection homomorphism of the corresponding long exact sequence.

On the other hand, we have that $E^1_{p,q}=\tilde{E}^1_{p,q}$ for $p\leq 0$. So we have group isomorphisms $\theta_{p,q}:E^1_{p,q}\longrightarrow \tilde{E}^1_{p,q}$ for every $p,q\in \Z$. We define $d^1_{p,q}=\theta^{-1}_{p-1,q}\circ\tilde{d}^1_{p,q}\circ\theta_{p,q}$ for all $p,q\in \Z$.

Let $n\in \Z$. We have a diagram
\[
\xymatrix@C=80pt@R=60pt{H_n(X_1,X_0)\ar@/^2pc/[rr]^-{\tilde{d}^1_{1,n-1}}\ar[r]^-{\partial}&H_{n-1}(X_0)\ar[r]^-{j_*}&H_{n-1}(X_0,X_{-1})\ar[d]^-{(j_*)^{-1}} \\
\bigoplus\limits_{x\in D_1}H_n(C_x^{X_{1}},\hat{C}^{X_{1}}_x)\ar[u]^-{\left(\biguplus\limits_{x\in D_1}i^x_n\right)_*}&\bigoplus\limits_{x\in D_1}\tilde{H}_{n-1}(\hat{C}^{X_{1}}_x)\ar[l]^{\left(\bigoplus\limits_{x\in D_1}\partial^x\right)^{-1}}\ar[u]^-{\left(\biguplus\limits_{x\in D_1}\tau^x_{n-1}\right)_*}\ar[r]_-{d^1_{1,n-1}}&H_{n-1}(X_0)}
\]
where the maps $(i^x_n)_*$ are defined as in \ref{lemma:isos phi} and where $\tau^x_{n-1}:C^f_{n-1}(\hat C_x^{X_1})\to C^f_{n-1}(X_0)$ are the inclusion homomorphisms.

It is clear that $\partial(i^x_n)_*=(\tau^x_{n-1})_*\partial^x$ for every $x\in D_1$ whenever $n\neq 1$ by the naturality of the long exact sequence. Moreover, since $\partial:H_1(X_1,X_0)\longrightarrow \tilde{H}_0(X_0)$ is the (range) restriction of $\partial:H_1(X_1,X_0)\longrightarrow H_0(X_0)$, then we have that $\partial(i^x_1)_*=(\tau^x_0)_*\partial^x$ for every $x\in D_1$ as well. Thus the left square in the last diagram commutes and therefore $d^1_{1,n-1}=\left(\biguplus\limits_{x\in D_1}\tau^x_{n-1}\right)_*$.

Now, let $[\sigma]\in \bigoplus\limits_{x\in D_1}\tilde{H}_{n-1}(\hat{C}^{X_{1}}_x)$. Then $[\sigma]=([\sigma_x])_{x\in D_1}$ with $[\sigma_x]\in \tilde{H}_{n-1}(\hat{C}^{X_{1}}_x)$ for each $x\in D_1$, and therefore $d^1_{1,n-1}([\sigma])=\sum\limits_{x\in D_1}[\sigma_x]$.

Let $p\geq 2$ and let $n\in \N$. Consider the following diagram
\[
\xymatrix@C=50pt@R=60pt{H_n(X_p,X_{p-1})\ar@/^2pc/[rr]^-{\tilde{d}^1_{p,n-p}}\ar[r]^-{\partial}&H_{n-1}(X_{p-1})\ar[r]^-{j_*}&H_{n-1}(X_{p-1},X_{p-2})\ar[d]^{(\phi_{n-1})_*}\\
\bigoplus\limits_{x\in D_p}H_n(C^{X_{p}}_x,\hat{C}^{X_{p}}_x)\ar[u]^-{\left(\biguplus\limits_{x\in D_p}i^x_n\right)_*}&\bigoplus\limits_{x\in D_p}\tilde{H}_{n-1}(\hat{C}^{X_{p}}_x)\ar[l]^{\left(\bigoplus\limits_{x\in D_p}\partial^x\right)^{-1}}\ar[dr]^-{d^1_{p,n-p}}\ar[u]^-{\left(\biguplus\limits_{x\in D_p}\tau^x_{n-1}\right)_*}&\bigoplus\limits_{y\in D_{p-1}}H_{n-1}(C^{X_{p-1}}_y,\hat{C}^{X_{p-1}}_y)\ar[d]^{\bigoplus\limits_{y\in D_{p-1}}\partial^y}\\
&&\bigoplus\limits_{y\in D_{p-1}}\tilde{H}_{n-2}(\hat{C}^{X_{p-1}}_y)}
\]
where, as above, the maps $(i^x_n)_*$ are defined as in \ref{lemma:isos phi} and where $\tau^x_{n-1}:C^f_{n-1}(\hat C_x^{X_p})\to C^f_{n-1}(X_{p-1})$ are the inclusion homomorphisms.

We have that
$$d^1_{p,n-p}=\left(\bigoplus\limits_{y\in D_{p-1}}\partial^y\right)(\phi_{n-1})_*j_*\partial\left(\biguplus\limits_{x\in D_p}i^x_n\right)_*\left(\bigoplus\limits_{x\in D_p}\partial^x\right)^{-1}$$

As before, the left square of the diagram commutes for every $n\in \N$. Therefore,
$$d^1_{p,n-p}=\left(\bigoplus\limits_{y\in D_{p-1}}\partial^y\right)(\phi_{n-1})_*j_*\left(\biguplus\limits_{x\in D_p}\tau^x_{n-1}\right)_*=\left(\bigoplus\limits_{y\in D_{p-1}}\partial^y\right)\left(\biguplus\limits_{x\in D_p}\phi_{n-1}j \tau^x_{n-1}\right)_*$$

Now we will find an explicit formula for $d^1_{p,n-p}$. Let $\sigma=(\sigma_x)_{x\in D_p}$ with $\sigma_x\in C^f_{n-1}(\hat{C}^{X_{p}}_x)$ for each $x\in D_p$. Then
$$\left(\biguplus\limits_{x\in D_p} \phi_{n-1}j \tau^x_{n-1}\right)(\sigma)=\sum\limits_{x\in D_p}\phi_{n-1}j\sigma_x=\sum\limits_{x\in D_p}\phi_{n-1}(\overline{\sigma_x}^{X_{p-2}})=\left(\sum\limits_{x\in D_p}\rho^y_{n-1}(\overline{\sigma_x}^{X_{p-2}})\right)_{y\in D_{p-1}}$$

Now, for each $x\in D_p$ we write $\sigma_x=\sum\limits_{i=1}^{l_x}a_i^x s^x_i$ with $l_x\in \N$, $a_i^x\in \Z$ and $s^x_i$ an $(n-1)$--chain in $\hat{C}^{X_p}_x$ for every $i\in \{1,\dots,l_x\}$. Then we have that
$$\rho^y_{n-1}(\overline{\sigma}^{X_{p-2}}_x)=\sum\limits_{i=1}^{l_x} a^x_i \rho^y_{n-1}(\overline{s^x_i}^{X_{p-2}})=\sum\limits_{s^x_i\ni y}a^x_i \overline{s^x_i}^y$$ for every $x\in D_p$.

Then,
$$\left(\biguplus\limits_{x\in D_p} \phi_{n-1}j \tau^x_{n-1}\right)(\sigma)=\left(\sum\limits_{x\in D_p}\sum\limits_{s^x_i\ni y}a^x_i \overline{s^x_i}^y\right)_{y\in D_{p-1}}$$

Finally, by lemma \ref{lemma:partial}, we see that
\begin{align*}
d^1_{p,n-p}([\sigma])&=\left(\bigoplus\limits_{y\in D_{p-1}}\partial^y\right)\left(\left[\left(\sum\limits_{x\in D_p}\sum\limits_{s^x_i\ni y}a^x_i\overline{s^x_i}^y\right)_{y\in D_{p-1}}\right]\right)=\\
&=\left(\partial^y\left(\left[\sum\limits_{x\in D_p}\sum\limits_{s^x_i\ni y}a^x_i\overline{s^x_i}^y\right]\right)\right)_{y\in D_{p-1}}=\\
&=\left(\left[\sum\limits_{x\in D_p}\sum\limits_{s^x_i\ni y}a^x_i\sgn_{s^x_i}(y)(s^x_i-\{y\})\right]\right)_{y\in D_{p-1}}
\end{align*}
\end{proof}

\begin{rem}
It is not difficult to prove that the morphisms of all the pages of the spectral sequence of the previous theorem can be computed in the same way as those of the first page, since for all $p,q\in\Z$ the groups $E^r_{p,q}$, $r\in\N$, are subquotients of the group $E^1_{p,q}$ and the morphisms of the $r$--th page of the spectral sequence are induced by the exact couple obtained from the long exact sequences in homology associated to the topological pairs $(X_p,X_{p-1})$.
\end{rem}

Note also that one can develop a similar spectral sequence to compute relative homology groups.

The spectral sequence of the previous theorem will get a much simpler form in the case that the homology of $\hat{U}_x$ is concentrated in some degree for all $x\in X$. This leads to the following definition.

\begin{definition}\label{def:quasi}
Let $X$ be a finite $T_0$--space. We say that $X$ is \emph{quasicellular} if there exists an order preserving map $\rho:X\longrightarrow \N_0$, which will be called \emph{quasicellular morphism for $X$}, such that
\begin{enumerate}[(1)]
\item The set $\{x\in X:\rho(x)=n\}$ is an antichain for every $n\in \N_0$. 
\item For every $x\in X$, the reduced homology of $\hat{U}_x$ is concentrated in degree $\rho(x)-1$.
\end{enumerate}
\end{definition}

Note that if $X$ is a quasicellular finite $T_0$--space, $\rho$ is a quasicellular morphism for $X$ and $x,y\in X$, then $x<y$ implies that $\rho(x)<\rho(y)$.

Note also that cellular spaces are quasicellular and that $h$--regular posets are quasicellular. But both inclusions are strict since
the non-Hausdorff suspension of the discrete space of three points is quasicellular but neither cellular nor $h$--regular.

Also, since an $h$--regular poset is not necessarily graded (see \cite[example 2.4]{Min}) we conclude that a poset might be quasicellular and non-graded.

\begin{coro} \label{coro_quasicel}
Let $X$ be a quasicellular finite space and let $\rho$ be a quasicellular morphism for $X$. Let $C(X)=(C_n(X),d_n)_{n\in\Z}$ be the chain complex defined by
\begin{itemize}
\item $C_n(X)=\bigoplus\limits_{\rho(x)=n}\tilde{H}_{n-1}(\hat{U}_x)$ for each $n\in \N_0$ and $C_n(X)=0$ for $n<0$.
\item For each $n\in \Z$, $d_n$ is the group homomorphism $d^1_{n+1,-1}$ of theorem \ref{theo:principal} (applied for the filtration defined by $X_p=\{x\in X:\rho(x)\leq p-1\}$ for all $p\in\Z$).
\end{itemize}
Then, $H_n(X)=H_n(C(X))$ for all $n\in\N_0$.
\end{coro}

\begin{proof}
Consider the filtration $\{X_p\}_{p\in \Z}$ given by $$X_p=\{x\in X:\rho(x)\leq p-1\}$$ for each $p\in \Z$. Clearly, the filtration $\{X_p\}_{p\in \Z}$ is induced by antichains, and thus theorem \ref{theo:principal} applies. Since $\rho$ is a quasicellular morphism for $X$, $\hat{C}_x^{X_p}=\hat{U}_x$  and hence it follows that $q=-1$ is the only non-trivial row of the first page of the spectral sequence, and therefore, the homology groups of $X$ are the homology groups of the chain complex determined by the groups and group homomorphisms on this row. By theorem \ref{theo:principal}, this chain complex is precisely $C(X)$.
\end{proof}

As a corollary of \ref{coro_quasicel}, we obtain theorem 3.7 of \cite{Min}:

\begin{theo}[Minian]
Let $X$ be a cellular space and let $C(X)$ be its cellular chain complex \cite[definition 3.6]{Min}. Then $H_p(X)=H_p(C(X))$ for each $p\in \Z$.
\end{theo}

\section{Applications}

\subsection*{Computation of homology groups of posets}

\ 

\ 

Just as a first and simple example consider the following:

\begin{ex} [Non-Hausdorff suspension]
Let $X$ be a finite $T_0$--space. 

Consider the filtration $\{X_p\}_{p\in \Z}$ of $\mathbb{S}X$ given by $X_0=U_+$ and $X_1=\mathbb{S}X$. Note that this filtration is induced by antichains. Since $U_+$ is contractible and $\hat{U}_+=X$, by theorem \ref{theo:principal} we have a bigraded spectral sequence $(E,d)$ that converges to $H_*(\mathbb{S}X)$ whose first page is
\[
\xymatrix@R=3pt{&&&&\\
&&\vdots&\vdots&\vdots\\
&&0&\tilde{H}_2(X)&0\\
&&0&\tilde{H}_1(X)&0\\
&&\Z\ar[l]+<-10pt,-10pt>;[rrr]+<0pt,-10pt>_>{p} \ar[d]+<-15pt,-10pt>;[uuuu]+<-15pt,0pt>^>{q} &\tilde{H}_0(X)&0&\\
&&\vdots&\vdots&\vdots\\
}
\] 

It is not hard to see that $d^1_{1,0}:\tilde{H}_0(X)\longrightarrow \Z$ is trivial. Hence, $H_0(\mathbb{S}X)=\Z$ and $H_n(\mathbb{S}X)=\tilde{H}_{n-1}(X)$ for all $n\in \N$.
\end{ex}

In the next example we will apply theorem \ref{theo:principal} to compute the homology groups of a more complicated finite space, which turns out to be a finite model of the real projective plane. This finite space was constructed by Barmak and Minian in \cite{BM2}.

\begin{ex}[Finite model of the real projective plane] \label{ex_projective_plane}
Let $X$ be the finite $T_0$--space whose Hasse diagram is
\[
\xymatrix@R=60pt{
j\bullet\ar@{-}[d]\ar@{-}[drr]\ar@{-}[drrrrr]&&k\bullet\ar@{-}[dl]\ar@{-}[d]\ar@{-}[drrrr]&&l\bullet\ar@{-}[dlll]\ar@{-}[d]\ar@{-}[dr]&&m\bullet\ar@{-}[d]\ar@{-}[dll]\ar@{-}[dllllll]\\
d\bullet\ar@{-}[d]\ar@{-}[drrr]&e\bullet\ar@{-}[dl]\ar@{-}[drr]&f\bullet\ar@{-}[dll]\ar@{-}[drrrr]&&g\bullet\ar@{-}[dllll]\ar@{-}[drr]&h\bullet\ar@{-}[dll]\ar@{-}[dr]&i\bullet\ar@{-}[d]\ar@{-}[dlll]\\
a\bullet&&&b\bullet&&&c\bullet\\
}
\] 

Consider the filtration $\{X_p\}_{p\in \Z}$ of $X$ given by $X_0=F_a$, $X_1=F_a\cup\{h,i\}$ and $X_2=X$. Note that this filtration is induced by antichains.

By \ref{theo:principal} there exists a bigraded spectral sequence $\{E^r,d_r\}_{r\in \N}$ that converges to $H_*(X)$.

Let $Z_a=E^1_{0,0}=H_0(F_a)\cong \Z$. Note that $Z_a$ is generated by $[a]$.

Observe that $E^1_{1,0}=\tilde{H}_0(\hat{F}_h)\oplus\tilde{H}_0(\hat{F}_i)=Z_h\oplus Z_i$, where $Z_h=\tilde{H}_0(\hat{F}_h)\cong \Z$ and $Z_i=\tilde{H}_0(\hat{F}_i)\cong \Z$. Also note that $Z_h$ is generated by $[l]-[j]$ and $Z_i$ is generated by $[m]-[k]$

Similarly, since $\hat{F}_b$ and $\hat{F}_c$ are finite models for $S^1$, we have that $E^1_{2,0}=\tilde{H}_1(\hat{F}_b)\oplus\tilde{H}_1(\hat{F}_c)=Z_b\oplus Z_c$, where $Z_b=\tilde{H}_1(\hat{F}_b)\cong \Z$ and $Z_c=\tilde{H}_1(\hat{F}_c)\cong \Z$. In this case we see that $Z_b$ and $Z_c$ are generated by 
$$g_0=[d,j]+[j,h]+[h,l]+[l,e]+[e,k]+[k,i]+[i,m]+[m,d]$$ and $$g_1=[f,j]+[j,h]+[h,l]+[l,g]+[g,m]+[m,i]+[i,k]+[k,f]$$
respectively. Now, it is easy to see that the first page of our spectral sequence is, in fact, a chain complex:
\[
\xymatrix@C=40pt{
\cdots&0\ar[l]& Z_a\ar[l]& Z_h\oplus Z_i\ar[l]_-{d_{1,0}^1=\alpha}& Z_b\oplus Z_c\ar[l]_-{d_{2,0}^1=\beta}&0\ar[l]&\cdots\ar[l]}
\]

Using theorem \ref{theo:principal} it is clear that $\alpha=0$. On the other hand, a quick calculation shows that $\beta(g_0)=([l]-[j],[m]-[k])$ and $\beta(g_1)=([l]-[j],[k]-[m])$. It follows that $$E^2_{1,0}= Z_h\oplus Z_i/\Ima \beta\cong\Z_2.$$

Thus, $H_0(X)=\Z$, $H_1(X)=\Z_2$ and $H_n(X)=0$ for $n\geq 2$.

It is interesting to observe that the chain complex above is much simpler than the simplicial chain complex of $\mathcal{K}(X)$.
\end{ex}

\subsection*{Homological morse theory of posets}

\ 

\ 

G. Minian introduced in \cite{Min} a discrete version of Morse theory for posets and developed a homological variant of this theory showing that given a homologically admissible Morse matching on a cellular poset $X$, the homology groups of $X$ can be computed from a chain complex which in degree $p$ consists of the free abelian group generated by the critical points of $X$ of degree $p$. Using our techniques we will give here a generalization of his result with a completely different and more conceptual proof.

We begin by recalling some definitions from \cite{Min}.

Let $X$ be a finite poset and let $\mathcal{H}(X)$ be its Hasse diagram. 

Let $M$ be a matching on $\mathcal{H}(X)$ and let $\mathcal{H}_M(X)$ be the graph obtained from $\mathcal{H}(X)$ by reversing the orientations of the edges which are not in $M$. We say that $M$ is a \emph{Morse matching} if the graph $\mathcal{H}_M(X)$ is acyclic.

Given a Morse matching $M$ on $\mathcal{H}(X)$ we define the \emph{critical points} of $X$ as the points of $X$ which are not incident to any edge of $M$.

We say that an edge $(a,b)$ of the Hasse diagram of $X$ is \emph{homologically admissible} if the space $\hat{U}_b-\{a\}$ is acyclic.

We say that a matching $M$ on $\mathcal{H}(X)$ is \emph{homologically admissible} if every edge of $M$ is homologically admissible.

For our proof we need some lemmas. The first of them is one of the keys of our proof.

\begin{lemma} \label{lemma_Morse_1}
Let $X$ be a poset. Let $b$ be a maximal point of $X$. Suppose there exists $a\in X$ such that $\hat{F}_a=\{x\in X \tq a<x\}=\{b\}$ and such that the edge $(a,b)$ of the Hasse diagram of $X$ is homologically admissible. Then $H_n(X,X-\{a,b\})=0$ for all $n\in\Z$.
\end{lemma}

\begin{proof}
Clearly, $a$ is upbeat point of $X$ an hence $X-\{a\}$ is a strong deformation retract of $X$. Thus, $H_n(X,X-\{a\})=0$ for all $n\in\Z$.

On the other hand, applying \ref{prop:+ hom C_x} we obtain that
$$H_n(X-\{a\},X-\{a,b\})=H_{n-1}(\hat{C}^{X-\{a\}}_b)=\tilde H_{n-1}(\hat{U}_b-\{a\})=0$$
for all $n\in \Z$ since the edge $(a,b)$ is homologically admissible.

Then the result follows from the long exact sequence in homology of the triple $(X,X-\{a\},X-\{a,b\})$.
\end{proof}

\begin{lemma} \label{lemma_Morse_2}
Let $X$ be a quasicellular poset and let $\rho:X\to \N_0$ be its quasicellular morphism. Let $(a,b)$ be an homologically admissible edge of the Hasse diagram of $X$. Then $\rho(b)=\rho(a)+1$.
\end{lemma}

\begin{proof}
We have that $\tilde H_{n}(\hat{U}_b-\{a\})=0$ for all $n\in \Z$ since the edge $(a,b)$ is homologically admissible and $H_n(\hat U_b,\hat U_b-\{a\})=\tilde H_{n-1}(\hat{U}_a)$ by \ref{prop:+ hom C_x}. Then the result follows from the long exact sequence in homology of the pair $(\hat{U}_b,\hat{U}_b-\{a\})$.
\end{proof}

Now, we will prove that the homology groups of a quasicellular poset $X$ can be computed from a chain complex constructed from a homologically admissible Morse matching, generalizing theorem 3.14 of \cite{Min}. The idea of the proof is that, by lemma \ref{lemma_Morse_1}, under certain hypotheses a homologically admissible edge of a poset can be removed without altering the homology groups of the poset. Thus, we will prove that the edges of a homologically admissible Morse matching can be arranged in such a way that the hypotheses of lemma \ref{lemma_Morse_1} are satisfied and hence we will be able to remove all the edges of the matching to obtain that the homology groups of $X$ can be computed from the set of critical points.

To put this idea into work, we consider a suitable filtration of the poset $X$ and construct a spectral sequence which converges to the homology groups of $X$.

\begin{theo}
Let $X$ be a quasicellular poset and let $p:X\to \N_0$ be its quasicellular morphism. Let $M$ be a homologically admissible Morse matching on $\mathcal{H}(X)$. For $n\in\N_0$ let $A_n=\{x\in X \tq \textnormal{$x$ is a critical point of $X$ and $\rho(x)=n$}\}$. Let $(C_n,d_n)_{n\in\N_0}$ be the chain complex defined by $\displaystyle C_n=\bigoplus_{x\in A_n} \tilde H_{n-1}(\hat U_x)$ and where the differentials $d_n$ are defined as in theorem \ref{theo:principal}. Then the homology of $X$ coincides with the homology of $(C_n,d_n)_{n\in\N_0}$.
\end{theo}

\begin{proof}
For $n\in\N_0$ let $$X_n=\{x\in X \tq \rho(x)\leq n\} \cup \{z \tq (y,z)\in M \textnormal{ and } \rho(y)=n\}.$$
By \ref{lemma_Morse_2}, $X_{n-1}\subseteq X_n$ for all $n\in\N$ and since $X$ is a finite space it follows that $(X_n)_{n\in\N_0}$ is a filtration of $X$.

Let $X_{-1}=\varnothing$. We will compute now $H_i(X_n,X_{n-1})$ for all $i\in\Z$ and for all $n\in\N_0$. Fix $n\in\N_0$. Let $\mathcal{S}=\{(y,z)\in M \tq \rho(y)=n\}$. Suppose $\mathcal{S}\neq\varnothing$. We define a relation $\preceq$ in $\mathcal{S}$ as follows: $(a,b)\preceq (a',b')$ if and only if there exist $l\in\N$ and $(a_0,b_0),\ldots,(a_l,b_l)\in \mathcal{S}$ such that $(a_0,b_0)=(a,b)$, $(a_l,b_l)=(a',b')$ and $a_j\in \hat U_{b_{j+1}}$ for all $j\in\{0,\ldots,l-1\}$.

This relation is clearly reflexive and transitive. We will prove now that it is also antisymmetric. Suppose that $(a,b)$ and $(a',b')$ are distinct elements of $\mathcal{S}$ such that $(a,b)\preceq (a',b')$ and $(a',b')\preceq (a,b)$. Then there exist $l,m\in\N$, $(a_0,b_0),\ldots,(a_l,b_l)\in \mathcal{S}$ and $(a_0',b_0'),\ldots,(a_m',b_m')\in \mathcal{S}$ such that $(a_0,b_0)=(a,b)=(a'_m,b'_m)$, $(a_l,b_l)=(a',b')=(a'_0,b'_0)$, $a_j\in \hat U_{b_{j+1}}$ for all $j\in\{0,\ldots,l-1\}$ and $a'_k\in \hat U_{b'_{k+1}}$ for all $k\in\{0,\ldots,m-1\}$. In addition, we may suppose that $(a_j,b_j)\neq (a_{j+1},b_{j+1})$ for all $j\in\{0,\ldots,l-1\}$ and that $(a'_k,b'_k)\neq (a'_{k+1},b'_{k+1})$ for all $k\in\{0,\ldots,m-1\}$. And since $M$ is a matching, we obtain that $a_j\neq a_{j+1}$ for all $j\in\{0,\ldots,l-1\}$ and $a'_k\neq a'_{k+1}$ for all $k\in\{0,\ldots,m-1\}$. Now, note that $(a_j,b_{j+1})\in\mathcal{H}(X)-M$ for all $j\in\{0,\ldots,l-1\}$ since $a_j<b_{j+1}$, $X$ is quasicellular, $\rho(a_j)=n=\rho(b_{j+1})-1$ (by the previous lemma) and $M$ is a matching. In a similar way $(a'_k,b'_{k+1})\in\mathcal{H}(X)-M$ for all $k\in\{0,\ldots,m-1\}$. Thus, 
\begin{displaymath}
(a'_m,b'_m),(b'_m,a'_{m-1}),(a'_{m-1},b'_{m-1}),\ldots,(a'_0,b'_0),(b_l,a_{l-1}),(a_{l-1},b_{l-1}),\ldots,(a_1,b_1),(b_1,a_0)
\end{displaymath}
is a cycle in $\mathcal{H}_M(X)$ which entails a contradiction since $M$ is a Morse matching.

Then the relation $\preceq$ is antisymmetric and hence it is a partial order.

Let $N=\#\mathcal{S}$. Extending the partial order in $\mathcal{S}$ to a linear order we obtain that we may label the elements of  $\mathcal{S}$ as $(y_1,z_1),\ldots,(y_N,z_N)$ in such a way that if $(y_j,z_j)\preceq (y_k,z_k)$ then $j\leq k$.

Note that 
$$X_n=X_{n-1}\cup A_n \cup \bigcup_{j=1}^N \{y_j,z_j\}$$
(recall that $A_n$ was defined as $A_n=\{x\in X \tq \textnormal{$x$ is a critical point of $X$ and $\rho(x)=n$ }\}$). 

For $k\in\{0,1,\ldots,N\}$ let 
$$B_k=X_{n-1}\cup A_n \cup \bigcup_{j=1}^k \{y_j,z_j\}\ .$$
Hence, $B_0=X_{n-1}\cup A_n$ and $B_N=X_n$.

Let $r\in\{1,\ldots,N\}$. We claim that $\hat U^X_{z_r}=\hat U^{B_r}_{z_r}$. Indeed, let $x\in \hat U^X_{z_r}$. Then $x<z_r$ and thus $\rho(x)<\rho(z_r)=n+1$ by the previous lemma. If $x\in X_{n-1}\cup A_n$ then $x\in B_r$. If $x\notin X_{n-1}\cup A_n$ then there exists $s\in\{1,\ldots,N\}$ such that $x=y_s$. Thus, $(y_s,z_s)\preceq (y_r,z_r)$ and hence $s\leq r$. Then, $x=y_s\in B_r$. Thus, $\hat U^X_{z_r}\subseteq \hat U^{B_r}_{z_r}$. The other inclusion is trivial.

Hence, the edge $(y_r,z_r)$ is homologically admissible in $B_r$.

On the other hand, we claim that $\hat F^{B_r}_{y_r}=\{z_r\}$. Indeed, suppose that $x\in B_r$ satisfies that $x>y_r$. Hence, $\rho(x)>\rho(y_r)=n$ and thus $x=z_s$ for some $s\in\{1,\ldots,r\}$. But this implies that $y_r\in \hat U_{z_s}$ and hence $(y_r,z_r)\preceq (y_s,z_s)$. Then $r\leq s$ and thus $s=r$ and $x=z_r$. Hence, $\hat F^{B_r}_{y_r}\subseteq \{z_r\}$ while the other inclusion is trivial.

Now, since $\rho(z_r)=n+1$ and $B_r\subseteq X_n \subseteq \{x\in X \tq \rho(x)\leq n+1\}$ (by \ref{lemma_Morse_2}) it follows that $z_r$ is a maximal point of $B_r$. Hence, we are under the hypotheses of lemma \ref{lemma_Morse_1} and thus we obtain that $H_i(B_r,B_{r-1})=0$ for all $i\in\Z$.

Hence, we have proved that $H_i(B_r,B_{r-1})=0$ for all $i\in\Z$ and for all $r\in\{1,\ldots,N\}$. It follows that $H_i(B_N,B_0)=0$ for all $i\in\Z$. Note that this was done under the hypotesis $\mathcal{S}\neq\varnothing$, but holds trivially if $\mathcal{S}=\varnothing$.

Thus, by \ref{prop:+ hom C_x},
$$H_i(X_n,X_{n-1})=H_i(B_N,X_{n-1})\cong H_i(B_0,X_{n-1})\cong \left\{
\begin{array}{cl}
\displaystyle \bigoplus_{x\in A_n} \tilde H_{n-1}(\hat U_x) & \textnormal{if $i=n$} \\
0 & \textnormal{if $i\neq n$}
\end{array}
\right.$$
since $X$ is quasicellular.

From the filtration $(X_n)_{n\in\N_0}$ of $X$ we can contruct a spectral sequence in a similar way to the one in the proof of theorem \ref{theo:principal}, which will have in its first page a single nontrivial row and whose differentials can be computed in the same way as in \ref{theo:principal} by naturality of the long exact sequences since the isomorphisms $H_i(X_n,X_{n-1})\cong H_i(B_0,X_{n-1})$ are given by the inclusion maps. Thus, the result follows.
\end{proof}

\begin{rem}
The previous theorem might not hold if the space $X$ is not quasicellular even if for all $x\in X$ the homology of $\hat U_x$ is concentrated in some degree. For example, let $X$ be defined by the following Hasse diagram
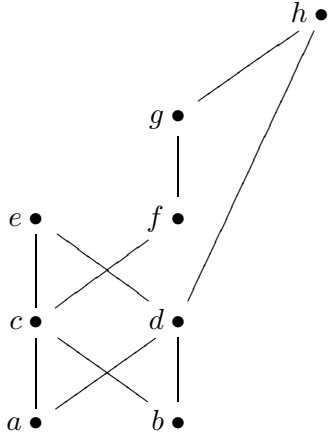
\[
\xymatrix{ & & h\bullet\phantom{h} \\ & g\bullet\phantom{g} \ar@{-}[ru] & \\ e\bullet\phantom{e} & f\bullet\phantom{f} \ar@{-}[u] & \\  c\bullet\phantom{c} \ar@{-}[u] \ar@{-}[ru] & d\bullet\phantom{d} \ar@{-}[lu] \ar@{-}[ruuu] & \\ a\bullet\phantom{a} \ar@{-}[u] \ar@{-}[ru] & b\bullet\phantom{b} \ar@{-}[lu] \ar@{-}[u] 
}
\] 
and let $M=\{(c,e),(d,h),(f,g)\}$. It is easy to verify that $M$ is a homologically admissible Morse matching. On the other hand $f$ and $g$ are beat points of $X$ and hence $X$ is homotopy equivalent to $X-\{f,g\}$ which is a finite model for $S^2$. But the set of critical points is $\{a,b\}$ an thus if $(C_n,d_n)_{n\in\N_0}$ is the chain complex of the previous theorem we obtain that $C_0=\Z\oplus\Z$ and $C_n=0$ for all $n\in\N$. Clearly, the homology groups of $(C_n,d_n)_{n\in\N_0}$ do not coincide with those of $X$.
\end{rem}

\subsection*{M\"obius function}

\ 

\

Now, we will apply our methods to obtain different formulas to compute the M\"obius function of a poset. Recall that the M\"obius function of a poset $P$ equals the number of chains of $P$ of odd cardinality minus the number of chains of $P$ of even cardinality (where the empty chain counts as a chain of even cardinality) and is denoted by $\mu(P)$. Clearly, $\mu(P)$ coincides with the reduced Euler characteristic of the order complex of $P$ and thus with the reduced Euler characteristic of $P$ (viewed as a topological space).

Recall also that the Euler characteristic of a finitely generated graded abelian group $G=(G_n)_{n\in\Z}$ is defined as $$\chi(G)=\sum_{n\in\Z} (-1)^{n}\rg(G_n)$$
and that the Euler characteristic of a finitely generated bigraded abelian group $E=(E_{p,q})_{p,q\in\Z}$ is defined as $$\chi(E)=\sum_{p,q} (-1)^{p+q}\rg(E_{p,q}).$$

Clearly, if $\{(E^r_{p,q})_{p,q\in \Z}\}_{r\in \N}$ is a bigraded spectral sequence such that $E^1$ is finitely generated then $\chi(E^r)=\chi(E^{r+1})$ for all $r\in\N$. Therefore, it follows that if $\{(E^r_{p,q})_{p,q\in \Z}\}_{r\in \N}$ converges to a graded abelian group $(G_n)_{n\in\Z}$ then $\chi(E^1)=\chi(G)$.

This simple idea suggests that our methods can be effectively applied to compute the M\"obius function of a finite poset in several different ways.

If $P$ is a finite poset, $h(P)$ will denote the height of $P$, that is the maximum of the cardinalities of the chains of $P$ minus one.

\begin{prop} \label{prop_Mobius_function}
Let $X$ be a finite $T_0$--space (or equivalently a finite poset) and let $V\subseteq X$ be an open subset. Then
$$\mu(X)=\mu(V)-\sum_{x\in X-V}\mu(\hat U_x).$$
\end{prop}

\begin{proof}
Let $X_0=V$. For $p\in \N$, let $X_p=\{x\in X-X_0 \tq h(U_x)<p \}$ and let $D_p=X_p-X_{p-1}$. Clearly $(X_p)_{p\in\N_0}$ is a filtration of $X$ which is induced by antichains. Hence, theorem \ref{theo:principal} applies and we obtain a spectral sequence $\{(E^r_{p,q})_{p,q\in \Z}\}_{r\in \N}$  that converges to $H_*(X)$ such that
\begin{itemize}
\item $E^1_{p,q}=0$ for every $p\leq -1$.
\item $E^1_{0,q}=H_q(X_0)$.
\item $E^1_{p,q}=\bigoplus\limits_{x\in D_p}\tilde{H}_{p+q-1}(\hat{C}^{X_p}_x)$ for $p\geq 1$.
\end{itemize}
Note that, if $p\in\N$ and $x\in D_p$ then $\hat{C}_x^{X_p}=\hat{U}_x$ since $V$ is an open set. Thus,
\begin{displaymath}
\begin{array}{rcl}
\mu(X) & = & \displaystyle \widetilde\chi(X)=\chi(X)-1=-1+\chi(E^1)=-1+\sum_{p\in\Z}\sum_{q\in\Z}(-1)^{p+q}\rg(E^1_{p,q})= \\
& = & \displaystyle -1+\sum_{q\in\Z}(-1)^{q}\rg(H_q(X_0))+\sum_{p\in\N}\sum_{q\in\Z}(-1)^{p+q}\rg\left(\bigoplus\limits_{x\in D_p}\tilde{H}_{p+q-1}(\hat{C}^{X_p}_x)\right) = \\ 
& = & \displaystyle -1 + \chi(X_0) + \sum_{p\in\N}\sum_{q\in\Z}(-1)^{p+q}\sum_{x\in D_p}\rg(\tilde{H}_{p+q-1}(\hat{U}_x)) = \\ 
& = & \displaystyle -1 + \chi(X_0) + \sum_{p\in\N}\sum_{x\in D_p}\sum_{q\in\Z}(-1)^{p+q}\rg(\tilde{H}_{p+q-1}(\hat{U}_x)) = \\
& = & \displaystyle \widetilde\chi(X_0) - \sum_{p\in\N}\sum_{x\in D_p}\widetilde\chi(\hat{U}_x) = \mu(V)-\sum_{p\in\N}\sum_{x\in D_p}\mu(\hat U_x) = \\
& = & \displaystyle \mu(V)-\sum_{x\in X-V}\mu(\hat U_x).
\end{array}
\end{displaymath}
\end{proof}

\begin{coro}
Let $X$ be a finite $T_0$--space (or equivalently a finite poset).
\begin{enumerate}
\item Let $A\subseteq X$ be a contractible open subspace. Then
$$\mu(X)=-\sum_{x\in X-A}\mu(\hat U_x).$$
\item Let $X_0$ be the set of minimal points of $X$. Then
$$\mu(X)=\#X_0-1-\sum_{x\in X-X_0}\mu(\hat U_x).$$
\end{enumerate}
\end{coro}

Note that the second formula of this corollary can also be deduced by partitioning the set of chains of $X$ according to the maximum element of each chain and noting that the chains of cardinality $n$ of $\hat U_x$ are in bijection with the chains of cardinality $n+1$ of $U_x$.

Now we will apply the previous results to give a different proof of theorem 6.1 of \cite{BjoWal}. To this end we need a couple of lemmas.

Recall that a subset $C$ of a poset $P$ is convex if $x<y<z$ and $x,z\in C$ imply $y\in C$.

\begin{lemma} \label{lemma_Fa_P-C}
Let $P$ be a finite poset and let $C$ be a convex subset of $P$. Let $a\in C$. Then 
$$\mu(\hat F_a -C)=\sum_{\substack{y\in C \\ y\geq a}} \mu (\hat F_y).$$
\end{lemma}

\begin{proof}
Applying \ref{prop_Mobius_function} with $X=(\hat F_a)^\textnormal{op}$ and $V=(\hat F_a -C)^\textnormal{op}$ yields
$$\mu((\hat F_a)^\textnormal{op})=\mu((\hat F_a -C)^\textnormal{op})-\sum_{x\in (\hat F_a)^\textnormal{op} \cap C}\mu (\hat U_x^X).$$
(note that $V$ is an open subset of $X$ since given $z\in V$ and $w\leq z$ we get that $w\leq z < a$, $a\in C$ and $z\notin C$, and thus $w\notin C$ since $C$ is convex).
Hence,
$$\mu(\hat F_a -C)=\mu(\hat F_a)+\sum_{y\in \hat F_a \cap C}\mu (\hat F_y^{\hat F_a}) = \mu(\hat F_a)+ \sum_{\substack{y\in C \\ y> a}} \mu (\hat F_y) = \sum_{\substack{y\in C \\ y\geq a}} \mu (\hat F_y).$$
\end{proof}

If $X$ and $Y$ are posets, the \emph{non-Hausdorff join} $X\oplus Y$ is the poset whose underlying set is the disjoint union $X\sqcup Y$ and whose ordering is defined keeping the ordering within $X$ and $Y$ and setting $x\leq y$ for every $x\in X$ and $y\in Y$ (cf. \cite{BarLN}).

\begin{lemma} \label{lemma_mu_join}
Let $X$ and $Y$ be posets. Then $\mu(X\oplus Y) = -\mu(X)\mu(Y)$.
\end{lemma}

\begin{proof}
Let $\mu_o(X)$ and $\mu_o(Y)$ denote the number of chains of odd cardinality of $X$ and $Y$ respectively and let $\mu_e(X)$ and $\mu_e(Y)$ denote the number of chains of even cardinality of $X$ and $Y$ respectively. Then
\begin{displaymath}
\begin{array}{rcl}
-\mu(X)\mu(Y) & = & -(\mu_o(X)-\mu_e(X))(\mu_o(Y)-\mu_e(Y)) = \\ 
& = & (\mu_o(X)\mu_e(Y)+\mu_e(X)\mu_o(Y))-(\mu_e(X)\mu_e(Y)+\mu_o(X)\mu_o(Y)) = \\
& = & \mu_o(X\oplus Y)-\mu_e(X\oplus Y)=\mu(X\oplus Y).
\end{array}
\end{displaymath}
\end{proof}

Now we state theorem 6.1 of \cite{BjoWal} and apply our methods to give an alternative proof of it.

\begin{theo}[Bj\"orner--Walker]
Let $P$ be a finite poset and let $C$ be a convex subset of $P$. Then
$$\mu(P)=\mu(P-C)+\sum_{\substack{x,y\in C \\ x\leq y}} \mu(\hat U_x) \mu(\hat F_y).$$
\end{theo}

\begin{proof}
We consider $P$ as a finite $T_0$--space. Let $X_0=P-C$. For $p\in \N$, let $X_p=X_0\cup\{x\in C \tq h(U_x)< p\}$ and let $D_p=X_p-X_{p-1}$.
Clearly $(X_p)_{p\in\N_0}$ is a filtration of $P$ which is induced by antichains. Hence, theorem \ref{theo:principal} applies and we obtain a spectral sequence $\{(E^r_{p,q})_{p,q\in \Z}\}_{r\in \N}$  that converges to $H_*(P)$ such that
\begin{itemize}
\item $E^1_{p,q}=0$ for every $p\leq -1$.
\item $E^1_{0,q}=H_q(X_0)$.
\item $E^1_{p,q}=\bigoplus\limits_{x\in D_p}\tilde{H}_{p+q-1}(\hat{C}^{X_p}_x)$ for $p\geq 1$.
\end{itemize}
Then, proceeding as in the proof of \ref{prop_Mobius_function} we obtain that
\begin{displaymath}
\mu(P) = \mu(P-C)-\sum_{p\in\N}\sum_{x\in D_p}\mu(\hat{C}^{X_p}_x).
\end{displaymath}

Note that $\hat{C}^{X_p}_x=\hat U_x \oplus \hat F_x^{P-C}$ for $x\in D_p$. Hence, from \ref{lemma_Fa_P-C} and \ref{lemma_mu_join} we obtain that
$$\mu(\hat{C}^{X_p}_x)=-\mu (\hat U_x)\mu (\hat F_x^{P-C})=-\mu (\hat U_x) \sum_{\substack{y\in C \\ y\geq x}} \mu (\hat F_y).$$
Thus,
\begin{displaymath}
\begin{array}{rcl}
\mu(P) & = & \displaystyle \mu(P-C)-\sum_{p\in\N}\sum_{x\in D_p}\mu(\hat{C}^{X_p}_x) = \mu(P-C)+\sum_{p\in\N}\sum_{x\in D_p} \left( \mu (\hat U_x) \sum_{\substack{y\in C \\ y\geq x}} \mu (\hat F_y) \right) = \\ 
& = & \displaystyle \mu(P-C)+\sum_{x\in C} \left( \mu (\hat U_x) \sum_{\substack{y\in C \\ y\geq x}} \mu (\hat F_y) \right) = 
\mu(P-C)+\sum_{\substack{x,y\in C \\ x\leq y}} \mu(\hat U_x) \mu(\hat F_y).
\end{array}
\end{displaymath}
\end{proof}

Clearly, this technique can be applied to obtain and prove many different formulas of this type for computing the M\"obius function of posets.

\section{Further generalizations and applications}

In this section we will analyse some generalizations of the results of section \ref{main_section} and we will show how these generalizations are useful in applications.

First of all, it is easy to see that the definitions of $f$--chain complex and $f$--homology groups given in section \ref{sect_prelim} and their relative and reduced versions can be generalized to Alexandroff T$_0$--spaces. Clearly, the $f$--homology groups will be isomorphic to the singular homology groups in this more general setting as well \cite{McC}.

Secondly, lemmas \ref{lemma:isos phi} and \ref{lemma:partial} and proposition \ref{prop:+ hom C_x} remain true for Alexandroff $T_0$--spaces. Moreover, theorem \ref{theo:principal} also holds for Alexandroff $T_0$--spaces provided that the filtration $\mathcal{F}=\{X_p:p\in \Z\}$ of $X$ is induced by antichains and satisfies the following condition
\begin{displaymath}
\begin{array}{cl}
\textnormal{($\ast$)} & \textnormal{For all $n\in\N$ there exists $m\in\N$ such that $\tilde{H}_{n}(\hat{C}^{X_p}_x)=0$ for all $p\geq m$ and } \\ & \textnormal{for all $x\in D_p=X_p-X_{p-1}$.}
\end{array}
\end{displaymath}
This condition guarantees that the spectral sequence will converge and is satisfied, for example, if the filtration $\mathcal{F}$ consists of only a finite number of distinct spaces (this is applied in \ref{theo:principal_cover_version}), but is also valid in a broad range of situations (such as that of \ref{coro_quasicel_relativo}).

Also, the following relative version of \ref{theo:principal} holds.

\begin{theo} \label{theo_principal_generalizado}
Let $X$ be an Alexandroff $T_0$--space and let $A\subseteq X$ be a subspace. Let $\mathcal{F}=\{X_p:p\in \Z\}$ be a filtration of $X$ with $X_0=A$ such that $\mathcal{F}$ is induced by antichains. For each $p\in \N$, let $D_p=X_p-X_{p-1}$. Suppose that the filtration $\mathcal{F}$ satisfies condition $(\ast)$ above.

Then there is a spectral sequence $\{(E^r_{p,q})_{p,q\in \Z},(d^r_{p,q})_{p,q\in \Z}\}_{r\in \N}$ that converges to $H_*(X,A)$ such that:
\begin{itemize}
\item $E^1_{p,q}=0$ for every $p\leq 0$.
\item $E^1_{p,q}=\bigoplus\limits_{x\in D_p}\tilde{H}_{p+q-1}(\hat{C}^{X_p}_x)$ for $p\geq 1$.
\item The morphisms $d^1_{p,q}:E^1_{p,q}\longrightarrow E^1_{p-1,q}$ are defined in the following way:
\begin{itemize}
\item If $p\leq 1$ and $q\in \Z$, then $d^1_{p,q}$ is the trivial homomorphism.
\item If $p\geq 2$ and $q\leq -p$, then $d^1_{p,q}$ is the trivial homomorphism.
\item If $p\geq 2$ and $q\geq 1-p$, then $d^1_{p,q}:\bigoplus\limits_{x\in D_p}\tilde{H}_{p+q-1}(\hat{C}^{X_p}_x)\longrightarrow\bigoplus\limits_{y\in D_{p-1}}\!\!\tilde{H}_{p+q-2}(\hat{C}^{X_{p-1}}_y)$ is defined by
$$d^1_{p,q}\left(\left(\left[\sum\limits_{i=1}^{l_x}a^x_i s^x_i\right]\right)_{x\in D_p}\right)=\left(\left[\sum\limits_{x\in D_p}\sum\limits_{s^x_i\ni y}a^x_i\sgn_{s^x_i}(y)(s^x_i-\{y\})\right]\right)_{y\in D_{p-1}}$$ 
where for every $x\in D_p$, $l_x\in \N$, and for every $i=\{1,\dots,l_x\}$, $a^x_i\in \Z$ and $s^x_i\in \tilde{C}_{p+q-1}(\hat{C}^{X_p}_x)$.
\end{itemize}
\end{itemize}
\end{theo}

The proof of this theorem is similar to that of \ref{theo:principal} and will be omitted.

Also, the definition of quasicellular space (\ref{def:quasi}) can be generalized as follows.

\begin{definition}\label{def:quasi_relativo}
Let $X$ be a locally finite $T_0$--space and let $A\subseteq X$ be a subspace. We say that $(X,A)$ is a \emph{relative quasicellular pair} if $A$ is open in $X$ and there exists an order preserving map $\rho:X-A\longrightarrow \N_0$, which will be called \emph{quasicellular morphism for $(X,A)$}, such that
\begin{enumerate}[(1)]
\item The set $\{x\in X-A:\rho(x)=n\}$ is an antichain for every $n\in \N_0$. 
\item For every $x\in X-A$, the reduced homology of $\hat{U}_x^X$ is concentrated in degree $\rho(x)-1$.
\end{enumerate}
\end{definition}

And corollary \ref{coro_quasicel} can be generalized accordingly.

\begin{coro} \label{coro_quasicel_relativo}
Let $(X,A)$ be a relative quasicellular pair and let $\rho$ be a quasicellular morphism for $(X,A)$. For each $n\in \N_0$, let $J_n=\{x\in X-A:\rho(x)=n\}$.

Let $C(X,A)=(C_n(X,A),d_n)_{n\in\Z}$ be the chain complex defined by
\begin{itemize}
\item $C_n(X,A)=\bigoplus\limits_{x\in J_n}\tilde{H}_{n-1}(\hat{U}_x)$ for each $n\in \N_0$ and $C_n(X,A)=0$ for $n<0$.
\item For each $n\in \Z$, $d_n$ is the group homomorphism $d^1_{n+1,-1}$ of theorem \ref{theo_principal_generalizado} (applied for the filtration defined by $X_0=A$ and $X_n=X_{n-1}\cup J_{n-1}$ for all $n\in\N$).
\end{itemize}
Then, $H_n(X,A)=H_n(C(X,A))$ for all $n\in\N_0$.

If, in addition, $A$ is contractible, then $\widetilde H_n(X)=H_n(C(X,A))$ for all $n\in\N_0$.
\end{coro}

\begin{ex} \label{ex_S2_v_S1}
Consider the poset $Z$ defined by the following Hasse diagram.
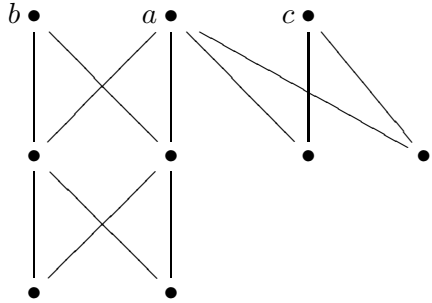
\[
\xymatrix@R=40pt{
b\bullet\phantom{b}\ar@{-}[d]\ar@{-}[dr] & a\bullet\phantom{a}\ar@{-}[dl]\ar@{-}[d]\ar@{-}[dr]\ar@{-}[drr] & c\bullet\phantom{c}\ar@{-}[d]\ar@{-}[dr] \\
\bullet\ar@{-}[d]\ar@{-}[dr] & \bullet \ar@{-}[dl]\ar@{-}[d] & \bullet & \bullet \\
\bullet & \bullet \\
}
\]
Note that $Z$ is not a cellular poset since $\hat U_a$ does not have the homology of any sphere. Hence, theorem 3.7 of \cite{Min} does not apply. Indeed, $Z$ is not even a quasicellular poset since $\hat U_a$ does not have homology concentrated in any degree.

However, $(Z,U_a)$ is a relative quasicellular pair with quasicellular morphism $\rho:Z-U_a \to \N_0$ defined by $\rho(b)=2$ and $\rho(c)=1$.

Hence, by \ref{coro_quasicel_relativo}, the homology groups of $(Z,U_a)$ can be computed with the following chain complex
\[
\xymatrix@C=40pt{
\cdots \ar[r] & 0 \ar[r] & \Z \ar[r]^{d_2} & \Z \ar[r] & 0 \ar[r] & \cdots}
\]

Applying the formulas for the differentials stated in \ref{theo:principal}, it is easy to check that $d_2=0$. Thus, $H_2(Z,U_a)\cong H_1(Z,U_a)\cong \Z$ and $H_n(Z,U_a)=0$ for all $n\in\Z-\{1,2\}$. And since $U_a$ is contractible we obtain that $\widetilde H_2(Z)\cong \widetilde H_1(Z)\cong\Z$ and $\widetilde H_n(Z)=0$ for all $n\in\Z-\{1,2\}$.

Indeed, $\mathcal{K}(Z)$ is homeomorphic to $S^2\vee S^1$.
\end{ex}

We turn now to a different generalization of our results which concerns group actions on locally finite T$_0$--spaces. From now on let $G$ be a group and let $Y$ be a locally finite and T$_0$ $G$--space. Clearly, for all $n\in\Z$ the group $G$ acts on $C_n^f(Y)$ since for every chain $C$ of $Y$ the set $gC$ is a chain in $Y$. Thus, the action of $G$ on $Y$ induces a $\Z[G]$--module structure in $C_n^f(Y)$ for all $n\in\Z$. Moreover, with this structure the differentials $d_n^f$ turn out to be morphisms of $\Z[G]$--modules.

We are particularly interested in coverings of locally finite T$_0$--spaces. Barmak and Minian gave in \cite{BM3} a simple description of the regular coverings of locally finite T$_0$--spaces in terms of colorings. We will recall briefly the basic points of this description from a categorical perspective. For the original and comprehensive exposition we refer the reader to \cite{BM3}.

A locally finite T$_0$--space (or equivalently a locally finite poset) $X$ can be regarded as a category in the standard way, that is, the objects of this category are the elements of $X$ and every relation $x\leq y$ in $X$ gives a unique arrow from $x$ to $y$. Also, if $G$ is a group, we can think of $G$ as a category with a unique object in the usual way. With these interpretations, if $G$ is a group and $X$ is a locally finite T$_0$--space, an \emph{admissible $G$--coloring of $X$} is simply a functor from $X$ to $G$. We say that an admissible $G$--coloring $c$ of $X$ is \emph{connected} if the induced functor $\overline{c}:X[\textnormal{Mor}(X)^{-1}]\to G$ is a full functor (here $X[\textnormal{Mor}(X)^{-1}]$ denotes the localization of the category $X$ with respect to the set of all the arrows of $X$). Note that an admissible $G$--coloring of a poset $X$ induces a labelling of the edges of the Hasse diagram of $X$ with elements of $G$.

If $c$ is an admissible connected $G$--coloring of $X$ the poset $E(c)$ is defined in \cite{BM3} as $E(c)=\{(x,g)\tq x\in X \textnormal{ and }g\in G\}$ with the relations $(x,g)\preceq (y,g\ldotp c(x\leq y))$ whenever $x\leq y$ in $X$. In this case the projection $p:E(c)\to X$ is a regular covering of $X$ with Deck transformation group isomorphic to $G$. Also, the universal cover of a locally finite T$_0$--space $X$ can be constructed as $E(c)$ where $c$ is a suitable admissible and connected $\pi_1(X)$--coloring of $X$ \cite{BM3}.

Now, we will combine these results of Barmak and Minian with the techniques developed in this article to obtain more applications together with generalizations and improvements on some results of \cite{BM4}.

Let $X$ be a locally finite T$_0$--space, let $c$ be an admissible connected $\pi_1(X)$--coloring of $X$ corresponding to the universal cover and let $E=E(c)$ be the universal cover of $X$. Let $G=\pi_1(X)$. Note that $E$ is a locally finite T$_0$ $G$--space with action given by $h\ldotp(x,g)=(x,hg)$ for $h\in G$ and $(x,g)\in E$. Let $\mathcal{F}=\{X_p:p\in \Z\}$ be a filtration of $X$ which is induced by antichains. Clearly, the filtration $\mathcal{F}$ induces a filtration $\mathcal{F}_E=\{X_p\times G\tq p\in\Z\}$ of $E$ which is also induced by antichains.

Note that if $A\subseteq X$ and we give to $A\times G$ the subspace topology with respect to $E$ then $U_{(x,g)}^{A\times G}$ is isomorphic to $U_x^A$ for all $x\in A$ and $g\in G$. Also, for all $x\in A$ and $g,h\in G$, if $g\neq h$ then $U_{(x,g)}^{A\times G}\cap U_{(x,h)}^{A\times G}=\varnothing$. Similar results hold for $\hat U_{(x,g)}^{A\times G}$, $F_{(x,g)}^{A\times G}$, $\hat F_{(x,g)}^{A\times G}$, $C_{(x,g)}^{A\times G}$ and $\hat C_{(x,g)}^{A\times G}$.

Thus, if the filtration $\mathcal{F}$ satisfies condition $(\ast)$ above then the filtration $\mathcal{F}_E$ also satisfies that condition. Hence, we can construct a spectral sequence which converges to the homology groups of $E$. Not surprisingly, the groups and morphisms of this spectral sequence will be determined by those of the spectral sequence corresponding to the filtration $\mathcal{F}$ of $X$.

To sum up, applying the locally finite version of theorem \ref{theo:principal} to the filtration $\mathcal{F}_E$ of $E$ one obtains the following result which gives a spectral sequence that converges to the homology groups of the universal cover of a finite T$_0$--space.

\begin{theo} \label{theo:principal_cover_version}
Let $X$ be a locally finite $T_0$--space and let $\{X_p:p\in \Z\}$ be a filtration of $X$ which is induced by antichains and which satisfies condition $(\ast)$. For each $p\in \N$, let $D_p=X_p-X_{p-1}$.

Let $c:X\to \pi_1(X)$ be an admissible and connected $\pi_1(X)$--coloring of $X$ corresponding to the universal cover and let $\widetilde X$ be the universal cover of $X$.

Then there is a spectral sequence $\{(E^r_{p,q})_{p,q\in \Z},(d^r_{p,q})_{p,q\in \Z}\}_{r\in \N}$ that converges to $H_*(\widetilde X)$ such that:
\begin{itemize}
\item $E^1_{p,q}=0$ for every $p\leq -1$.
\item $E^1_{0,q}=H_q(X_0\times \pi_1(X))$ \ \textnormal{(}$X_0\times \pi_1(X)$ is given the subspace topology with respect to $E(c)$\textnormal{)}.
\item $E^1_{p,q}=\bigoplus\limits_{x\in D_p}\bigoplus\limits_{g\in \pi_1(X)}\tilde{H}_{p+q-1}(\hat{C}^{X_p}_x)$ for $p\geq 1$.
\item The morphisms $d^1_{p,q}:E^1_{p,q}\longrightarrow E^1_{p-1,q}$ are defined in the following way:
\begin{itemize}
\item If $p\leq 0$ and $q\in \Z$, then $d^1_{p,q}$ is the trivial homomorphism.
\item If $p=1$ and $q\in\N_0$, then $d^1_{p,q}:\bigoplus\limits_{x\in D_1}\bigoplus\limits_{g\in \pi_1(X)}\tilde{H}_q(\hat{C}^{X_1}_x)\longrightarrow H_q(X_0\times\pi_1(X))$ is induced by the inclusion maps
$\hat{C}^{X_1}_{x}\times \pi_1(X) \longrightarrow X_0 \times \pi_1(X)$, with $x\in D_1$.
\item If $p\geq 1$ and $q\leq -p$, then $d^1_{p,q}$ is the trivial homomorphism.
\item If $p\geq 2$ and $q\geq 1-p$, then the differential 
$$d^1_{p,q}:\bigoplus\limits_{x\in D_p}\bigoplus\limits_{g\in \pi_1(X)}\tilde{H}_{p+q-1}(\hat{C}^{X_p}_x)\longrightarrow\bigoplus\limits_{y\in D_{p-1}}\bigoplus\limits_{g\in \pi_1(X)}\tilde{H}_{p+q-2}(\hat{C}^{X_{p-1}}_y)$$ 
is defined as follows. For each $x'\in D_p$ and $g'\in\pi_1(X)$ let 
$$\incl_{x',g'}:\tilde{H}_{p+q-1}(\hat{C}^{X_p}_{x'}) \longrightarrow \bigoplus\limits_{x\in D_p}\bigoplus\limits_{g\in \pi_1(X)}\tilde{H}_{p+q-1}(\hat{C}^{X_p}_x)$$
denote the canonical inclusion in the corresponding coordinate of the direct sum. In a similar way, for each $y'\in D_{p-1}$ and $h'\in\pi_1(X)$ let 
$$\incl'_{y',h'}:\tilde{H}_{p+q-2}(\hat{C}^{X_{p-1}}_{y'}) \longrightarrow \bigoplus\limits_{y\in D_{p-1}}\bigoplus\limits_{g\in \pi_1(X)}\tilde{H}_{p+q-2}(\hat{C}^{X_{p-1}}_y)$$
denote the canonical inclusion. Also, if $a\leq b$ we will denote $c(a\leq b)^{-1}$ by $c(b\leq a)$.

With this notations, the differential $d^1_{p,q}$ is the group homomorphism defined by
\begin{displaymath}
d^1_{p,q}\left(\incl_{x,g}\left(\left[\sum\limits_{i=1}^{l}a_i s_i\right]\right)\right) = \displaystyle \sum_{y\in D_{p-1}\cap C_x}\!\!\incl'_{y,g\ldotp c(x\leq y)} \left(\left[\sum\limits_{s_i\ni y}a_i\sgn_{s_i}(y)(s_i-\{y\})\right]\right)
\end{displaymath}
for all $x\in D_p$ and $g\in\pi_1(X)$, where $l\in\N$ and where for every  $i=\{1,\dots,l\}$, $a_i\in \Z$ and $s_i\in \tilde{C}_{p+q-1}(\hat{C}^{X_p}_x)$.
\end{itemize}
\end{itemize}
\end{theo}

\begin{ex}
Consider the poset $X$ of example \ref{ex_projective_plane} which is a finite model of the projective plane. In examples 3.7 and 4.5 of \cite{BM3}, Barmak and Minian computed its fundamental group by means of the following $\Z_2$--coloring
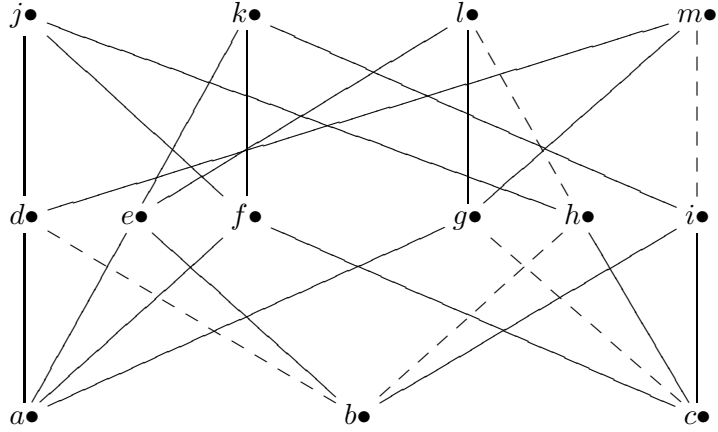
\[
\xymatrix@R=60pt{
j\bullet\ar@{-}[d]\ar@{-}[drr]\ar@{-}[drrrrr]&&k\bullet\ar@{-}[dl]\ar@{-}[d]\ar@{-}[drrrr]&&l\bullet\ar@{-}[dlll]\ar@{-}[d]\ar@{--}[dr]&&m\bullet\ar@{--}[d]\ar@{-}[dll]\ar@{-}[dllllll]\\
d\bullet\ar@{-}[d]\ar@{--}[drrr]&e\bullet\ar@{-}[dl]\ar@{-}[drr]&f\bullet\ar@{-}[dll]\ar@{-}[drrrr]&&g\bullet\ar@{-}[dllll]\ar@{--}[drr]&h\bullet\ar@{--}[dll]\ar@{-}[dr]&i\bullet\ar@{-}[d]\ar@{-}[dlll]\\
a\bullet&&&b\bullet&&&c\bullet\\
}
\] 
where the dashed edges are labelled with the generator of $\Z_2$ and where the remaining edges are labelled with the identity element of $\Z_2$.

We will apply the techniques we developed in this article to compute the second homotopy group of this space (and hence that of the real projective plane). Let $\{X_p\}_{p\in \Z}$ be the filtration of $X$ defined in example \ref{ex_projective_plane} and let $\widetilde X$ be the universal cover of $X$.

By theorem \ref{theo:principal_cover_version} and by the computations of \ref{ex_projective_plane} we obtain a spectral sequence which converges to the homology groups of $\widetilde X$. The first page of this spectral sequence is, in fact, a chain complex which, using the same notations as in example \ref{ex_projective_plane}, turns out to be the following
\[
\xymatrix@C=22pt{
\cdots&0\ar[l]& Z_a\oplus Z_a \ar[l] & Z_h\oplus Z_h\oplus Z_i\oplus Z_i\ar[l]_-{d_{1,0}^1=\overline\alpha} & Z_b\oplus Z_b\oplus Z_c\oplus Z_c\ar[l]_-{d_{2,0}^1=\overline\beta} & 0 \ar[l]&\cdots\ar[l]}
\]
By the previous theorem we obtain that 
\begin{displaymath}
\begin{array}{lcl}
\overline\alpha([l]-[j],0,0,0) & = & (-[a],[a]) \\
\overline\alpha(0,[l]-[j],0,0) & = & ([a],-[a]) \\
\overline\alpha(0,0,[m]-[k],0) & = & (-[a],[a]) \\
\overline\alpha(0,0,0,[m]-[k]) & = & ([a],-[a])
\end{array}
\end{displaymath}
and that
\begin{displaymath}
\begin{array}{lcl}
\overline\beta(g_0,0,0,0) & = & (0,[l]-[j],[m]-[k],0) \\
\overline\beta(0,g_0,0,0) & = & ([l]-[j],0,0,[m]-[k]) \\
\overline\beta(0,0,g_1,0) & = & ([l]-[j],0,[k]-[m],0) \\
\overline\beta(0,0,0,g_1) & = & (0,[l]-[j],0,[k]-[m]) 
\end{array}
\end{displaymath}
A simple computation shows that $E^2_{2,0}=\ker \overline{\beta} \cong \Z$. Thus, $\pi_2(X)\cong\pi_2(\widetilde X)\cong H_2(\widetilde X)\cong\Z$.
\end{ex}

%

With the techniques developed above we will give a generalization of Hurewicz's theorem for which we need some previous results.

\begin{lemma} \label{lemma_subdiv_X_A}
Let $X$ be a finite T$_0$--space and let $A$ be a subspace of $X$. Then
\begin{enumerate}
\item $A'$ is an open subset of $X'$ and $(X-A)'\subseteq X'-A'$.
\item The inclusion $i:(X-A)'\to X'-A'$ is a weak homotopy equivalence.
\end{enumerate}
\end{lemma}

\begin{proof}
(1) Follows easily from the definition of barycentric subdivision of a poset.

(2) Let $\sigma\in X'-A'$ and let $\eta=\sigma\cap(X-A)$. Then $\eta\in (X-A)'$ and
$$i^{-1}(U_\sigma)=\{\tau\tq\textnormal{$\tau$ is a chain of $X-A$ and $\tau\subseteq \sigma$}\}=U_\eta$$
which is contractible. Hence, by McCord's theorem (\cite[theorem 6]{McC}), $i$ is a weak homotopy equivalence.
\end{proof}

The following result is a generalization of Hurewicz's theorem for locally finite T$_0$--spaces which will be applied to obtain a similar generalization for regular CW-complexes.

\begin{theo} \label{theo_pi_2}
Let $X$ be a locally finite and connected T$_0$--space. Suppose that there exist a nonempty subset $A\subsetneq X$ and $n\in\N$ with $n\geq 2$ such that all the connected components of $A$ are $n$--connected and such that the inclusion of each connected component of $X-A$ in $X$ induces the trivial morphism between the fundamental groups.

If $l\in\N$ is such that $2\leq l\leq n$ and $H_j(X)=0$ for all $2\leq j\leq l-1$ then $\pi_j(X)=0$ for all $2\leq j\leq l-1$ and 
$\pi_l(X)=H_l(X)\otimes\Z[\pi_1(X)]$. In particular, $\pi_2(X)=H_2(X)\otimes\Z[\pi_1(X)]$.
\end{theo}

\begin{proof}
By the previous lemma, taking the barycentric subdivisions of the posets $X$ and $A$ we may suppose that $A$ is an open subspace of $X$ and that $X$ is a cellular poset and hence that $(X,A)$ is a relative quasicellular pair. 

Let $\rho$ be a quasicellular morphism for $(X,A)$. Consider the filtration $\{X_p\}_{p\in \Z}$ of $X$ given by 
\begin{displaymath}
X_p=\left\{
\begin{array}{ll}
\varnothing & \textnormal{if $p\leq -1$} \\
A \cup \{x\in X-A:\rho(x)\leq p\} & \textnormal{if $p\geq 0$}
\end{array}
\right.
\end{displaymath}
Now, by remark 4.2 of \cite{BM3} there exists a $\pi_1(X)$--coloring $c$ of $X$ corresponding to the universal cover such that $c(a,b)$ is the identity element of $\pi_1(X)$ for all $a,b\in X-A$ with $a \leq b$. Note that $A$ and $\{x\in X-A:\rho(x)=0\}$ are disjoint open subsets and that $\{x\in X-A:\rho(x)=0\}$ is a discrete subspace.

Applying \ref{theo:principal_cover_version} yields a spectral sequence that converges to $H_\ast(\tilde X)$ whose first page is
\[
\xymatrix@R=5pt{&&&\\
&\vdots&\vdots&\vdots\\
&\displaystyle \bigoplus_{g\in\pi_1(X)}\!\!\! H_2(A) & 0 & 0 \\
&\displaystyle \bigoplus_{g\in\pi_1(X)}\!\!\! H_1(A) & 0 & 0 \\
&\displaystyle \bigoplus_{g\in\pi_1(X)}\!\!\! H_0(X_0) \ar[l]+<0pt,-25pt>;[rrr]+<20pt,-25pt>_>{p} \ar[ddd]+<-35pt,-10pt>;[uuuu]+<-35pt,0pt>^>{q} & 
\displaystyle\bigoplus\limits_{x\in D_1}\bigoplus\limits_{g\in \pi_1(X)}\!\!\!\tilde{H}_{0}(\hat{C}^{X_1}_x) \ar[l]_-{d^1_{1,0}} & 
\displaystyle\bigoplus\limits_{x\in D_2}\bigoplus\limits_{g\in \pi_1(X)}\!\!\!\tilde{H}_{1}(\hat{C}^{X_2}_x) \ar[l]_-{d^1_{2,0}} & 
\ldots \ar[l]_-{d^1_{3,0}} \\
\\
&  0 & 0 & 0 \\
&\vdots&\vdots&\vdots \\
}
\] 
Since $c(a,b)$ is the identity element of $\pi_1(X)$ for all $a,b\in X-A$ with $a \leq b$, by \ref{theo:principal_cover_version}, we obtain that  $d^1_{j,0}=\bigoplus\limits_{g\in \pi_1(X)} d_j$ for all $j\geq 2$, where $d_j$ is the differential defined in corollary \ref{coro_quasicel_relativo} for the relative quasicellular pair $(X,A)$. Hence $E^2_{p,0}\cong H_p(X,A)$ for all $p\geq 2$. And since $A$ is $n$--connected, from the spectral sequence above we get that $H_k(\tilde X)\cong \bigoplus\limits_{g\in \pi_1(X)} H_k(X,A)\cong \bigoplus\limits_{g\in \pi_1(X)} H_k(X)\cong H_k(X)\otimes \Z[\pi_1(X)]$ for all $2\leq k \leq n$. The result follows.
\end{proof}

The following remark analyses the hypothesis on the existence of the subspace $A$ of the previous theorem giving a necessary condition for it to hold.

\begin{rem}
If $X$ is a locally finite and connected T$_0$--space which satisfies the hypotheses of theorem \ref{theo_pi_2} then the fundamental group of $X$ is free.

Indeed, let $X$ be a locally finite T$_0$--space such that there exist $A\subsetneq X$ and $n\in\N$ with $n\geq 2$ such that all the connected components of $A$ are $n$--connected and such that the inclusion of each connected component of $X-A$ in $X$ induces the trivial morphism between the fundamental groups. For simplicity, assume that $X$ is a finite T$_0$--space and that $A$ is connected. With some technical considerations, the proof of this case can be extended to the general case.

By lemma \ref{lemma_subdiv_X_A}, taking the barycentric subdivisions of the posets $X$ and $A$ we may suppose that $A$ is an open subspace of $X$. Let $\{B_1,\ldots,B_r\}$ be the set of connected components of $X-A$. Note that $B_j$ is closed in $X$ for all $j\in\{1,\ldots,r\}$. Since $X$ is connected, the subsets $B_j$ can not be open in $X$ and thus, for each $j\in\{1,\ldots,r\}$ there exist $a_j\in X-B_j$ and $b_j\in B_j$ such that $a_j<b_j$. But if $a_j$ belongs to some $B_k$ with $k\neq j$ then $b_j\in B_k$ (since $B_k$ is closed in $X$) and thus $B_k\cap B_j\neq \varnothing$ which can not be possible. Hence, $a_j\in A$ for all $j\in\{1,\ldots,r\}$. Moreover, we may suppose that $(a_j,b_j)$ is an edge of the Hasse diagram of $X$. Now let $D$ be the subdiagram of the Hasse diagram of $X$ defined as the union of the Hasse diagram of $A$ with the Hasse diagrams of all the subsets $B_j$ for $j\in\{1,\ldots,r\}$ together with the edges $(a_j,b_j)$ for all $j$. Let $Z$ be the finite T$_0$--space whose Hasse diagram is $D$. Applying the description of the fundamental group of finite spaces given in \cite[section 2.4]{BarLN} it is easy to prove that the inclusion of $D$ in the Hasse diagram of $X$ induces the trivial map between the fundamental groups of $Z$ and $X$ (for any choice of the basepoint of $Z$).

Now, applying the proof of theorem 4.4 of \cite{BM3} we get a description of $\pi_1(X)$ by generators and relations, where the set of generators is the set of edges of the Hasse diagram of $X$ which are not in $D$. But since $A$ is an open subset of $X$, every monotonic edge-path in $X$ has at most one of such edges. Hence, the relations of the aforementioned description of $\pi_1(X)$ either identify two generators or identify a generator with the identity element of $\pi_1(X)$ (cf. \cite[corollary 4.9]{BM3}). Thus, $\pi_1(X)$ is a free group.
\end{rem}

Note that the converse of the implication stated in the previous remark does not hold. Indeed, take $X$ as a finite model of $S^1\times S^2$. Then $\pi_1(X)\cong \Z$, $\pi_2(X)\cong \Z$ and $H_2(X)\cong \Z$. Hence $\pi_2(X)\not\cong H_2(X)\otimes\Z[\pi_1(X)]$ and thus $X$ does not satisfy the hypotheses of theorem \ref{theo_pi_2}.

The following is a simple example of application of \ref{theo_pi_2}.

\begin{ex}
Consider the poset $Z$ defined in example \ref{ex_S2_v_S1}. We wish to compute $\pi_2(Z)$.

Clearly, the inclusion of each connected component of $Z-U_a$ in $Z$ induces the trivial morphism between the fundamental groups, and since $U_a$ is contractible, the previous theorem applies and yields
$$\pi_2(Z)=H_2(Z)\otimes\Z[\pi_1(Z)].$$

Now, in example \ref{ex_S2_v_S1} we obtained that $H_2(Z)\cong\Z$.

On the other hand, the fundamental group of $Z$ is easy to compute. Considering the open sets $U_a\cup U_c$ and $U_b$ and applying the van Kampen theorem yields 
$$\pi_1(Z) \cong \pi_1(U_a\cup U_c) \cong \pi_1(U_c\cup\{a\}) \cong \Z\ .$$
Here, the first isomorphism holds since the map $\pi_1(U_b\cap (U_a\cup U_c))\to \pi_1(U_a\cup U_c)$ induced by the inclusion is trivial as it can be factorized through $\pi_1(U_a)$. And the second isomorphism holds since $U_c\cup\{a\}$ can be obtained from $U_a\cup U_c$ by removing beat points.

Thus, $\pi_2(Z)=\Z[\Z]$ (indeed, as we mentioned in \ref{ex_S2_v_S1}, $\mathcal{K}(Z)$ is homeomorphic to $S^2\vee S^1$).
\end{ex}

Finally, we will apply \ref{theo_pi_2} to obtain a generalization of Hurewicz theorem for regular CW-complexes.

\begin{theo}
Let $X$ be a connected regular CW-complex. Suppose that there exist a nonempty subcomplex $A\subsetneq X$ and $n\in\N$ with $n\geq 2$ such that all the connected components of $A$ are $n$--connected and such that the inclusion of each connected component of $X-A$ in $X$ induces the trivial morphism between the fundamental groups.

If $l\in\N$ is such that $2\leq l\leq n$ and $H_j(X)=0$ for all $2\leq j\leq l-1$ then $\pi_j(X)=0$ for all $2\leq j\leq l-1$ and 
$\pi_l(X)=H_l(X)\otimes\Z[\pi_1(X)]$. In particular, $\pi_2(X)=H_2(X)\otimes\Z[\pi_1(X)]$.
\end{theo}

\begin{proof}
Taking barycentric subdivision we may suppose that $X$ is the geometric realization of a simplicial complex $K$ and that $A$ is the geometric realization of a full subcomplex $B$ of $K$. Then $\mathcal{X}(K)$ is a connected and locally finite T$_0$--space and $\mathcal{X}(B)$ is $n$--connected.

We will prove now that the inclusion $j:\mathcal{X}(K)-\mathcal{X}(B) \to \mathcal{X}(K)$ induces the trivial morphism between the fundamental groups for any choice of the basepoint of $\mathcal{X}(K)-\mathcal{X}(B)$. Let $L$ be the subcomplex of $K$ defined by $V_L=V_K-V_B$ and $S_L=\{\sigma\in S_K \tq \sigma\subseteq V_L\}$.
Note that $X-A=|K|-|B|=\bigcup\limits_{v\in V_L} \textnormal{st}(v)$ (here $\textnormal{st}(v)$ denotes the open star of the vertex $v$) and since $L$ is a full subcomplex of $K$ we obtain that $|L|$ is a strong deformation retract of $|K|-|B|$ \cite[p. 124]{Spa}.

Let $i:L\to K$ be the inclusion. Since the inclusion of each connected component of $X-A$ in $X$ induces the trivial morphism between the fundamental groups we obtain that $|i|:|L| \to |K|$ induces the trivial morphism between the fundamental groups for any choice of the basepoint of $|L|$. Hence, $(\mathcal{X}(i))_\ast:\pi_1(\mathcal{X}(L))\to \pi_1(\mathcal{X}(K))$ is the trivial morphism for any choice of the basepoint of $\mathcal{X}(L)$. Let $i': \mathcal{X}(L) \to \mathcal{X}(K)-\mathcal{X}(B)$ be the inclusion map. Consider the following commutative diagram
\begin{displaymath}
\xymatrix{\pi_1(\mathcal{X}(K)-\mathcal{X}(B)) \ar[r]^-{j_\ast} & \pi_1(\mathcal{X}(K)) \\ \pi_1(\mathcal{X}(L)) \ar[u]^{i'_\ast} \ar[ru]_-{(\mathcal{X}(i))_\ast=0} }
\end{displaymath}
Let $r:\mathcal{X}(K)-\mathcal{X}(B) \to \mathcal{X}(L)$ be defined by $r(\sigma)=\sigma\cap V_L$ for all $\sigma\in \mathcal{X}(K)-\mathcal{X}(B)$. Note that $r$ is well defined since $B$ is a full subcomplex of $K$. And since $r$ is order-preserving, it is continuous. Clearly, $ri'=\id_{\mathcal{X}(L)}$. Also, $i'r(\sigma)\leq \sigma$ for all $\sigma\in \mathcal{X}(K)-\mathcal{X}(B)$ and the equality holds for all $\sigma\in \mathcal{X}(L)$. Hence $i'r\simeq \id_{\mathcal{X}(K)-\mathcal{X}(B)}\ \textnormal{rel }\mathcal{X}(L)$. Thus, the map $r$ is a strong deformation retract.

Hence, $i'_\ast$ is an isomorphism and thus $j_\ast:\pi_1(\mathcal{X}(K)-\mathcal{X}(B)) \to \pi_1(\mathcal{X}(K))$ is the trivial morphism if we choose a basepoint which belongs to $\mathcal{X}(L)$. But since $r$ is a strong deformation retract, any point of $\mathcal{X}(K)-\mathcal{X}(B)$ belongs to the same connected component of some point of $\mathcal{X}(L)$. Hence, the inclusion $j:\mathcal{X}(K)-\mathcal{X}(B)\to \mathcal{X}(K)$ induces the trivial morphism between the fundamental groups for any choice of the basepoint of $\mathcal{X}(K)-\mathcal{X}(B)$.

Therefore, theorem \ref{theo_pi_2} applies and the result follows.
\end{proof}

\begin{rem}
With these ideas into consideration, one can provide a proof of the previous theorem without the use of finite spaces. Indeed, let $X$ be a connected CW-complex and let $C\subseteq X$ be a connected subcomplex such that the morphism $i_\ast:\pi_1(C)\to\pi_1(X)$ induced by the inclusion is trivial. Let $p:\widetilde{X}\to X$ be the universal cover of $X$ and let $\widetilde{C}$ be a connected component of $p^{-1}(C)$. Let $p'=p|_{\widetilde{C}}:\widetilde{C}\to C$. Then $p'$ is a covering and $(p')_\ast(\pi_1(\widetilde{C}))=\ker(i_\ast)=\pi_1(C)$. Thus, $(p')_\ast:\pi_1(\widetilde{C})\to\pi_1(C)$ is an isomorphism and therefore $p':\widetilde{C}\to C$ is a homeomorphism. Applying this result to each connected component of $X-A$ and computing cellular homology one can verify that $H_k(\widetilde{X},p^{-1}(A))=\bigoplus\limits_{g\in\pi_1(X)}H_k(X,A)$ for all $k$. Since $A$ is $n$-connected, each connected component of $p^{-1}(A)$ is $n$-connected and the result follows.

Note that this proof is essentially the same as the proof of \ref{theo_pi_2}. Clearly, the finite space approach shows explicitly the key ideas of the result which were also needed for its formulation.
\end{rem}

\end{document}